\documentclass{article}

\usepackage{svg}
\usepackage{amsthm}
\usepackage{amsmath}
\usepackage{amssymb}
\usepackage{mathrsfs}
\usepackage{graphicx}
\usepackage{hyperref}
\usepackage{bm}
\usepackage{dsfont}
\usepackage{enumitem}
\usepackage{amsfonts}
\usepackage{textcomp}
\usepackage{MnSymbol}
\usepackage{geometry}
\usepackage[english]{babel}
\usepackage[T1]{fontenc}

\usepackage{shorttoc}

\usepackage{caption}
\usepackage{graphicx}
\graphicspath{ {./images/} }
\usepackage{pst-knot}
\usepackage{tikz}
\usetikzlibrary{patterns}
\usepackage[utf8]{inputenc}

\newtheorem{theo}{Theorem}[section]

\newtheorem{lem}[theo]{Lemma}

\newtheorem{definition}[theo]{Definition}

\newtheorem{proposition}[theo]{Proposition}

\newtheorem{corollary}[theo]{Corollary}

\newtheorem{ex}[theo]{Example}

\newtheorem{nonex}[theo]{Non-Example}

\newtheorem{conjecture}[theo]{Conjecture}

\title{Strong Property (T) and relatively hyperbolic groups}
\author{Hermès Lajoinie-Dodel \footnote{Hermès Lajoinie-Dodel, hermes.lajoinie@umontpellier.fr, IMAG, Univ Montpellier, CNRS, France}}
\date{December, 2023}

\begin{document}
\maketitle

\begin{center}
\begin{minipage}{0.8\textwidth}
\textsc{Abstract.} We prove that relatively hyperbolic groups do not have Lafforgue strong Property $(T)$ with respect to Hilbert spaces.

To do so we construct an unbounded affine representation of such groups, whose linear part is of polynomial growth of degree $2$. Moreover, this representation is proper for the metric of the coned-off graph.
\end{minipage}
\end{center}

\let\thefootnote\relax

\section{Introduction}

In 1967, Kazhdan defined in \cite{Kazdhanarticlefondateur} the \emph{Property (T)} for locally compact groups to show that a large class of lattices are finitely generated. One characterization of Property $(T)$, due to Delorme-Guichardet, is that any action by affine isometries on a Hilbert space has a fixed point. This definition enlightens the idea that actions of groups with Property $(T)$ are rigid.

In the same way, another characterization of Property $(T)$ given by Chatterji, Drutu, Haaglund in \cite{Tmedianviewpoint}, is that any action on a median space has bounded orbits. Examples of median spaces are $CAT(0)$ cube complexes, and in particular trees. However this fixed point property fails for general Gromov-hyperbolic spaces, which are a natural generalization of the geometry of trees. In fact, uniform lattices in $Sp(n,1)$ for $n\ge2$ and some random groups \cite{ZukKazdhanandRandomgroups},\cite{KotowskiKotowskirandomgroupsandT}, have Property $(T)$ and are hyperbolic.\\

In his work on the Baum-Connes conjecture, Vincent Lafforgue defined in 2007 in \cite{Lafforgue} a strengthening of Property $(T)$ called the \emph{strong Property $(T)$}. Instead of looking at unitary representations only, strong Property $(T)$ deals with representations with a small growth of the norm. Moreover, strong Property $(T)$ found various applications. The most important one is maybe in dynamics, as it was one of the steps in spectacular progresses on the Zimmer program \cite{BrownFisherHurtadoZimmer'sconjecturecocompact},\cite{BrownFisherHurtadoSlN}, which aims at understanding actions of higher rank lattices on manifolds.

In \cite{Lafforgue}, Lafforgue proved the following Theorem :

\begin{theo} \label{theoreme Lafforgue}\cite[Theorem 1.4]{Lafforgue}

Let $G$ be a locally compact group having the strong Property $(T)$. Let $X$ be a uniformly locally finite Gromov-hyperbolic graph equipped with an isometric action of $G$. Then every orbit for this action is bounded.
    
\end{theo}

In particular, looking at the action on its Cayley graph, infinite hyperbolic groups cannot have the strong Property $(T)$.

In \cite[Theorem 2.1]{Lafforgue}, Lafforgue also showed that $SL(3,\mathbb{R})$ has the strong Property $(T)$. Improving this result, de la Salle and de Laat proved in \cite{deLaatdelaSalle(T)renforcéepourgroupesdeLiederangsupérieurs} that connected higher rank Lie groups have the strong Property $(T)$, de la Salle proved also in \cite{delaSallePropTrenforcéepourlesréseauxderangsupérieurs} that this result is also true for higher rank lattices. Another kind of examples of groups with the strong Property $(T)$ are cocompact lattices of $\tilde{A_{2}}$-buildings (see \cite{lécureux2023strong}).\\

Higher rank lattices are well-known to enjoy other rigidity properties. We can mention the Margulis superrigidity theorem. According to Bader and Furman in \cite{BaderFurmanboundariesRigidityhyperbolicdesréseaux} and to Haettel in \cite{ThomasRigiditéHyperboliquedesréseauxderangsupérieur} independently, their actions on hyperbolic spaces are very rigid. In fact, an isometric action of a higher rank lattice on a hyperbolic space is either elliptic or parabolic and this is true without any futher assumption on the hyperbolic space. Since these groups are the most important examples with strong Property $(T)$, a natural question to ask is the following one.

\begin{conjecture} \label{conjecture : conjecture rigidité hyperbolique}
Let $G$ be a discrete countable group with strong Property $(T)$ and let $(X,d)$ be a Gromov-hyperbolic space. Any action of $G$ by isometries on $X$ is elementary.
\end{conjecture}

For example, it is an open question to know if the mapping class group of a closed oriented surface does not have the strong Property $(T)$ (this question is also open for Property $(T)$). However, these groups act on the curve graph, which is hyperbolic (see \cite{MasurMinsky}), but this graph is not locally finite, therefore Lafforgue's Theorem \ref{theoreme Lafforgue} does not apply. \\

In this article, we prove the following Theorem, which is a positive answer to Conjecture \ref{conjecture : conjecture rigidité hyperbolique} for a particular class of Gromov-hyperbolic spaces.

\begin{theo}\label{theorem:main result}
Let $G$ be a locally compact group having the strong Property $(T)$. Let $X$ be a \emph{uniformly fine} hyperbolic graph equipped with an isometric action of $G$. Then every orbit for this action is bounded.

\end{theo}

A graph is uniformly fine if for every edge the set of simplicial loops of a fixed length containing this edge is finite and the number of such loops does not depend on the edge. With this definition, we can see that uniformly fine graphs could be locally infinite.
Yet, locally finite graphs are uniformly fine, therefore Theorem \ref{theorem:main result} is a generalisation of Theorem \ref{theoreme Lafforgue}. This class of graphs are important in geometric group theory to study \emph{relatively hyperbolic groups}.

Relative hyperbolicity was defined by Gromov in 1987 in \cite{Gromovhypgroups}.  It is a geometric generalisation to hyperbolic groups for a larger class of groups.  The rough idea is that a group $G$ is hyperbolic relatively to a family of subgroups $\mathcal{P}$ if the geometry of $G$ is hyperbolic outside of $\mathcal{P}$. One characterization due to Bowditch of these groups (see \cite{Bowditchrelhyp}) is a particular action on a uniformly fine hyperbolic graph, which is unbounded in particular. Therefore Theorem \ref{theorem:main result} leads to the following corollary :

\begin{corollary}\label{corollaire : gp rel hyp n'ont pas T renforcée}

Let $G$ be a group hyperbolic relative to a family of subgroups $\mathcal{P}$ with infinite index in $G$. Then $G$ does not have the strong Property $(T)$.
\end{corollary}

We can also deduce the following rigidity result of groups with the strong Property $(T)$.

\begin{corollary}\label{corollary : propriété de rigidité sur morphisme vers groupe rel hyp}

Let $G$ be a group with the strong Property $(T)$. Let $\Gamma$ be a group hyperbolic relative to $(H_{1},...,H_{k})$, where $H_{i}$ for $1 \le i \le k$ are subgroups of $\Gamma$.

Let $ \varphi: G \to \Gamma$ be a group morphism, then $\varphi(G)$ is finite or is included in one of the conjugates of a $H_{i}$.
\end{corollary}

To prove Theorem \ref{theorem:main result}, we build a representation with polynomial growth on a Hilbert space without global non-zero fixed vector.

\begin{theo}\label{theorem:main result bis}

Let $G$ be a group hyperbolic relatively to a family of subgroups with infinite index in $G$.

Then, there exists a representation $\alpha$ of $G$ on a Hilbert space $H$, which is proper for the metric of the coned-off graph, such that the linear part $\pi$ of $\alpha$, is polynomial of degree $2$, i.e. there exists a length $l$ on $G$ such that, for all $g \in G$ :

$$ \vert\vert \pi(g) \lvert\lvert \leq P(l(g))$$

where $P \in \mathbb{R}[X]$ and $deg P = 2$.

\end{theo}

It is unknown whether this could be improved to a proper action if parabolic subgroups have the Haagerup Property. For instance, non-uniform lattices in $Sp(n,1)$ are hyperbolic relatively to nilpotent subgroups according to \cite[Lemma II.10.27]{BridsonHaefliger}. Therefore parabolics subgroups have the Haagerup Property, but the lattices have Property $(T)$.

The study of representations of polynomial growth and in particular uniformly bounded representations is related to the following conjecture.

\begin{conjecture}\label{conjecture : Shalom}(Shalom, see \cite[Conjecture 35]{Nowakbouquin})
    Every hyperbolic group acts properly by affine uniformly Lipschitz actions on a Hilbert space.
\end{conjecture}

In other words, if this conjecture is true, every hyperbolic group acts properly by affine transformations on a Hilbert space with a linear part of degree $0$. This result is clearly true for hyperbolic groups, which have Haagerup Property. For hyperbolic groups with Property $(T)$, only lattices in $Sp(n,1)$ are known to satisfy the conjecture, due to the work of Nishikawa \cite{nishikawa}.\\

Conjecture \ref{conjecture : Shalom} is not true for a large class of groups which generalizes the geometry of hyperbolic groups and includes relatively hyperbolic groups called \emph{acylindrically hyperbolic groups}. In fact, as noticed by Drutu and Mckay in \cite{drutu2023actions}, according to the work of Minasyan–Osin \cite{OsinMinasyan} and de Laat–de la Salle \cite{deLaatdelaSalleL2sepctralgap}, there exists an acylindrically hyperbolic group which statisfies the fixed point property for every uniformly bounded representation on $L^{p}$ and $\ell^{p}$ for $1< p < \infty$ and for $p=2$ in particular. This result implies that we have to look at polynomial representations of degree at least $1$ to have a chance to find a representation without non-zero fixed vectors on Hilbert spaces for acylindrically hyperbolic groups. \\

In this article, we will focus on representations on Hilbert spaces only. Yet another idea to generalize Property $(T)$ is to study \emph{isometric representations on Banach spaces}. In \cite{BaderfurmantsachikMonodBanachT}, Bader, Furman, Gelander and Monod proved that Property $(T)$ is equivalent to the fixed property for isometric representations on the Banach space $X=L^{1}([0,1])$. On the other hand, Yu proved in \cite{Yuactionaffinesurdeslp} that every hyperbolic group admits a proper action on an $\ell^{p}$-space for $p$ large enough. This result provides examples of groups with Property $(T)$, that do not satisfy the fixed point property on $L^{p}$ for $p$ large enough.

In the same way, Vincent Lafforgue defined in \cite{Lafforgue} the \emph{strong Banach Property $(T)$}, which looks at representations with a small growth of the norm on Banach spaces. In \cite{delaSallePropTrenforcéepourlesréseauxderangsupérieurs}, \cite{lécureux2023strong}, it is proved that examples of groups with the strong Property $(T)$ with respect to Hilbert spaces, also satisfy the strong Property $(T)$ for a large class of Banach spaces including $L^{p}$ spaces for $1<p<\infty$.

In \cite{ChatterjiDahmani}, Chatterji and Dahmani generalize Yu's result and proved that relatively hyperbolic groups admit an isometric action on an $\ell^{p}$ for $p$ large enough with an unbounded orbits. Therefore these groups do not have the strong Property $(T)$ with respect to the Banach spaces $L^{p}$, for $p$ large enough, but this question was still open for Hilbert spaces. In the rest of the text, the strong Property $(T)$ will always denote the strong hilbertian Property $(T)$.\\

\textbf{Structure of the article.} Section \ref{section : Propriété T } and Section \ref{section : Propriété T renforcée } are an overview on Property $(T)$ and strong Property $(T)$ with an emphasis on fixed points.

Section \ref{section : groupes relativement hyperbolic} is dedicated to the geometry of relatively hyperbolic groups. We give the proof of the existence of a particular geodesic triangle in a uniformly fine graph, which will be very useful in the rest of the paper. We give also a proof of Corollary \ref{corollary : propriété de rigidité sur morphisme vers groupe rel hyp}.

In Section \ref{section : le cas des arbres}, we give the proof of Theorem \ref{theorem:main result} in the simplest case, the case of locally infinite tree.

In Section \ref{section : Preuve groupes relativement hyperbolic }, we give the proof of Theorem \ref{theorem:main result} in the general case. We construct a subexponential representation on Hilbert spaces of relatively hyperbolic groups, which is unbounded. In the last paragraph, we will show that this representation is proper for the coned-off graph metric and prove Theorem \ref{theorem:main result bis}.\\

\textbf{Acknowledgements.}  The author is very grateful to Thomas Haettel for many helpful discussions, his availability and his careful reading of the article. The author would like to thank Pablo Montealegre, Jérémie Brieussel and Anthony Genevois for interesting discussions. The author would like to thank Amélie Vernay for her skills with linux and the installation of IPEdrawing.

\section{Property (T) and fixed points}\label{section : Propriété T }

This section is here to give a general and quick presentation on Property $(T)$ with an emphasis on fixed points. As a general reference on the topic, the reader could look at the monograph \cite{Bouquin(T)}.

In this definition Hilbert spaces are \emph{complex}.

\begin{definition}\label{Property (T)}\cite{Kazdhanarticlefondateur} (Kazhdan Property $(T)$)

Let $G$ be a topological group.

$G$ has Kazdhan Property $(T)$ if every unitary representation $(\pi,H)$ with almost invariant vectors has a non-zero invariant vector.

\end{definition}

We will present in this section another characterization of Property $(T)$ called Property $(FH)$, which is a fixed point result for affine representations on Hilbert spaces and discuss some of its consequences.

In this paragraph, Hilbert spaces are \emph{real}.

\begin{definition}(Property (FH))

A topological group $G$ has $\emph{Property (FH)}$ if every affine isometric action of $G$ on a real Hilbert space has a fixed point.

\end{definition}

The next theorem states that Property $(T)$ and Property $(FH)$ are equivalent for a large class of groups. The first point is due to Delorme \cite[Théorème V.1]{Delorme}, the second point is due to Guichardet \cite[Théorème 1]{Guichardet}.

\begin{theo} \label{theorem de Delorme-Guichardet}   (Delorme-Guichardet)

Let $G$ be a topological group.
\begin{itemize}
    \item If $G$ has Property $(T)$, then $G$ has Property $(FH)$.
    
    \item If $G$ is a $\sigma$-compact locally compact group with Property $(FH)$, then $G$ has Property $(T)$.

\end{itemize}
    
\end{theo}

The equivalence between Property $(T)$ and Property $(FH)$ is true in particular for all discrete countable groups.

\section{Strong Property $(T)$}\label{section : Propriété T renforcée }

In this section, we give the precise definition of the strong Property $(T)$ and discuss about it.

\subsection{Property $(T)$ in term of Kazhdan Projection}

In this subsection, we give a last equivalent definition of Property $(T)$ in terms of Kazhdan Projection. This definition was the main ingredient of the definition of Lafforgue's strong Property $(T)$.

\begin{definition}\label{definition:Kazhdan projection}
Let $G$ be a locally compact group.

A \emph{Kazhdan projection} for $G$, denoted by $p$ is an idempotent in $C^{\star}_{max}(G)$ such that for any unitary representation $(\pi,H)$ on a Hilbert space, $\pi(p)$ is the orthogonal projection on the subspace of $H$ of $G$-invariant vectors.

\end{definition}
The following Theorem is due to Akemann and Walker in \cite{akemann_walter_1981}.

\begin{theo}\cite{akemann_walter_1981} \label{Theorem : Prop (T) Kazhdan projection}
Let $G$ be a locally compact group.

The group $G$ has Property $(T)$ if and only if $G$ admits a Kazhdan projection.

\end{theo}

\subsection{Definition and fixed point properties}

In this section, we give the precise definition of the strong Property $(T)$ and discuss about it.

\begin{definition}\label{definition de longueur}

Let $G$ be a topological group.

A \emph{length function} on $G$ is a continuous function $l:G \longrightarrow \mathbb{R}^{+}$ such that :

\begin{itemize}
    \item for all $g \in G$, $l(g)=l(g^{-1})$,
    
    \item for all $g_{1}, g_{2}$, $l(g_{1}g_{2}) \le l(g_{1})+l(g_{2}).$

\end{itemize}
    
\end{definition}

In a finitely generated group, the word metric with respect to a generating set of the group is a particular example of a length function.

\begin{definition}\label{algebrelongueur}
Let $G$ be a locally compact group. If $l$ is a length function on $G$, we denote by $\mathcal{E}_{G,l}$ the set of continuous representations $\pi$ of $G$ on a Hilbert space $H$ such that $|| \pi(g) || \le e^{l(g)}$ for all $g \in G$. \\

We denote by $\mathcal{C}_{l}(G)$ the involutive Banach algebra, which is the completion of $C_{c}(G)$ for the norm $\vert\vert f\lvert\lvert=\sup_{(H,\pi)\in \mathcal{E}_{G,l}} \vert\vert \pi(f)\lvert\lvert$.

\end{definition}

In particular, for $l=0$, we have $\mathcal{C}_{0}(G)=C^{\ast}_{max}(G)$. In \cite{Lafforgue}, Lafforgue sets a definition of strong Property $(T)$, which looks like the characterization of Property $(T)$ in Theorem \ref{Theorem : Prop (T) Kazhdan projection}, for representations with a small exponential growth.

\begin{definition}
A locally compact group $G$ has \emph{strong Property $(T)$} if for every length function $l$, there exists $s>0$, such that for all $C\in \mathbb{R}_{+}$, there exists a real and selfadjoint projection $p$ in $\mathcal{C}_{sl+C}(G)$, such that for all representations $\pi \in \mathcal{E}_{G,sl+C}$, $\pi(g)$ is a projection on the $\pi(G)$-invariant vectors.

\end{definition}

Contrary to the case of Property $(T)$, the projection is not necessarily orthogonal.\\

According to de la Salle in \cite{delaSalleBanachTpoursl3}, the initial definition of strong Property $(T)$ of Lafforgue is equivalent to the following definition.

\begin{definition}\cite[Definition 1.1]{delaSalleBanachTpoursl3}\label{definition : T renforcée de la Salle}

A locally compact group $G$ has \emph{strong Property $(T)$} with respect to Hilbert spaces if for every length function $l$ on $G$ there is a sequence of compactly supported symmetric Borel measures $m_{n}$ on $G$ such that, for every Hilbert space $H$, there is a constant $t>0$ such that the following holds. For every strongly continuous representation $\pi$ of $G$ on $H$ satisfying $ \left\| \pi(g) \right\|_{B(X)} \le L e^{tl(g)} $ for some $L \in \mathbb{R}_{+}$, $\pi(m_{n})$ converges in the norm topology of $B(X)$ to a projection on the $\pi(G)$-invariant vectors of $X$.

\end{definition}

The definition of strong Property $(T)$ has huge consequences towards fixed point property for sub-exponential representations.

\begin{definition}\label{definition sous-exponentiel}

A representation $\pi: G \longrightarrow B(H)$ on a Hilbert space is \emph{sub-exponential} if there exists a length $l$  such that for all $\varepsilon>0$, there exists $C>0$ such that for all $g\in G$ :

$$ \left\| \pi(g) \right\| \le Ce^{ \varepsilon l(g)}.$$

A representation $\pi: G \longrightarrow B(H)$ on a Hilbert space is \emph{polynomial} if there exists a length $l$ and a polynomial $P$ such that for all $g\in G$ :

$$ \left\| \pi(g) \right\| \le P(l(g)).$$

A representation $\pi: G \longrightarrow B(H)$ on a Hilbert space is \emph{uniformly bounded} if there exists $K \in \mathbb{R}_{+}$ such that for all $g\in G$ :

$$ \left\| \pi(g) \right\| \le K.$$

\end{definition}

A polynomial representation is in particular sub-exponential.\\

To prove Theorem \ref{theorem:main result}, we will use the following Property of groups having the Strong Property $(T)$ :

\begin{proposition}\label{Proposition : fixed point for strong Property (T)}
Let $G$ be a locally compact group having the strong Property $(T)$.

\begin{itemize}
    \item If $G$ admits  a sub-exponential representation $\pi$ on an Hilbert space $H$ and the representation fixes a closed affine hyperplane, $H$ also admits a non-zero $G$-invariant vector.

    \item If $G$ admits an affine representation on an affine Hilbert space $H$, whose linear part is sub-exponential then $H$ admits a non-zero $G$-invariant vector.
    
\end{itemize}

\end{proposition}

\begin{proof}
Since $G$ has the strong Property $(T)$ and $\pi$ is sub-exponential, according to \ref{definition : T renforcée de la Salle} , there exists a sequence of compactly supported symmetric Borel measures $m_{n}$ such that $\pi(m_{n})$ converges to $p_{\pi}$ a projection on $G$-invariant vector.\\

Let us denote $\mathcal{H}$ the affine closed hyperplan fixed by $\pi$. Let $x_{0} \in \mathcal{H} $. Then : 

$$\left\| \pi(m_{n})(x_{0})-p_{\pi}(x_{0}) \right\| \le \left\| \pi(m_{n})-p_{\pi} \right\| \left\| x_{0}  \right\|. $$

Since the right hand converges to zero, $\pi(m_{n})(x_{0})$ converges to the $G$-invariant vector $p_{\pi}(x_{0})$.

Moerover since $\mathcal{H}$ is closed and $G$-invariant, $p_{\pi}(x_{0}) \in \mathcal{H}$. Since $\mathcal{H}$ is affine, we deduce that $p_{\pi}(x_{0})$ is non-zero and the proposition is proved. \\

To prove the second point, as explained in \cite[Proposition 5.6]{LafforgueTrenforcéebanachiqueettransforémedeFourierrapide}, every affine representation  $\alpha$ on an affine Hilbert space $H$ with basis $H_{0}$, can be realized as the restriction to the affine hyperplane $H_{0} \oplus \mathbb{C}$ of linear representation $\pi \oplus id_{\mathbb{C}}$. Moreover if the linear part of $\alpha$ is sub-exponential then $\pi$ will also be sub-exponential. According to the first point of the proposition, $\pi$ admits a non-zero invariant vector. With this vector, we get a non-zero invariant vector for $\alpha$.

\end{proof}

In particular, if a representation $(\pi,H)$ of a group $G$ with strong Property $(T)$ preserves a continuous linear form $l$, i.e 

$$ \forall g \in G, \forall \xi \in H , l(\pi(g)\xi)=l(\xi). $$

We can apply Proposition \ref{Proposition : fixed point for strong Property (T)} to the affine closed hyperplane $\{ \xi \in H, l(\xi)=1 \}$ and we can deduce the existence of a non-zero $G$-invariant vector.

\subsection{Examples of groups with the strong Property $(T)$}

In this subsection, we will briefly gives examples of groups with strong Property $(T)$.

The first result is due to Lafforgue in \cite{Lafforgue}.

\begin{theo}\label{Theorem: strong Property Tpour sl(3,R)}\cite[Thérorème 2.1]{Lafforgue}\

  $Sl_{3}(\mathbb{R})$  has the strong Property $(T)$.
\end{theo}

According to the work of Lafforgue, the strong Property $(T)$ passes from a group to its cocompact lattices. Therefore according to Lafforgue's Theorem \ref{theoreme Lafforgue}, connected simple Lie groups with real rank $1$ do not have the strong Property $(T)$. In fact some of these even do not have Kazhdan Property $(T)$ and the other contains hyperbolic cocompact lattice.\\

For Lie groups of rank at least $2$, the situation is completely different. De la Laat and de la Salle generalized Theorem \ref{Theorem: strong Property Tpour sl(3,R)}.

\begin{theo}\cite[Theorem A]{deLaatdelaSalle(T)renforcéepourgroupesdeLiederangsupérieurs}

Let $G$ be a connected simple Lie group with real rank at least $2$. Then $G$ has strong Property $(T)$.
\end{theo}

Contrary to Property $(T)$, when $\Gamma$ is a non-uniform lattice in $G$, representations which are not uniformly bounded do not behave well under induction. Therefore, we do not know if strong Property $(T)$ passes to lattices. Yet, de la Salle proved the following Theorem with different methods.

\begin{theo}\cite[Theorem 1.1]{delaSallePropTrenforcéepourlesréseauxderangsupérieurs}

 Every lattice in a higher-rank group has strong Property $(T)$.
\end{theo}

For example, strong Property $(T)$ for $Sl_{n}(\mathbb{R})$, $n\le 3$ is a corollary of this Theorem and was still an open question before.\\

The first example of groups with strong Property $(T)$, that do not come from algebraic groups over local fields is due to Lécureux, de la Salle and Witzel.

\begin{theo} \cite[Main Theorem.]{lécureux2023strong}

Let $X$ be an $\tilde{A}_{2}$-building, and let $\Gamma$ be an undistorted lattice on $X$. Then $\Gamma$ has strong Property $(T)$.

\end{theo}

\section{ Relatively hyperbolic groups and geodesic triangles }\label{section : groupes relativement hyperbolic}

To prove Theorem \ref{theorem:main result}, we will start with a group $G$ having strong Property $(T)$ and a uniformly fine graph $X$ equipped with an isometric action of $G$. We will build a representation of $G$ on a Hilbert, which follows hypothesis of Proposition \ref{Proposition : fixed point for strong Property (T)}.

In this part, we will describe some useful tools to work on relatively hyperbolic groups and we will prove the existence of a particular geodesic triangle between any three points that we will use in the following sections.

\subsection{Fine graphs and relatively hyperbolic groups}

Let us start by recalling some basic definitions and results.

Let $X$ be a simplicial graph.
To fix notations, for all $x,y \in X^{(0)}$, $[x,y]$ will denote a geodesic between $x$ and $y$. We will denote by $I(x,y)=\{a \in X^{(0)} ~|~ d(x,y)=d(x,a)+d(a,y) \} $ the geodesic interval between $x$ and $y$. For all $x,y,z \in  X^{(0)}$, $[x,y,z]$ will denote a geodesic triangle between $x$, $y$ and $z$ i.e. a choice of three geodesics between each of these three points. We denote by $X^{(1)}$ the set of unoriented edges of $X$. For $e \in X^{(1)}$, $e \subset I(x,y)$ will denote an edge on some geodesic between $x$ and $y$ and $e \subset [x,y]$ will denote an edge on the geodesic $[x,y]$. For $e\in X^{(1)}$, for $x\in X$, $x\in e$ will denote a vertex on the edge $e$. \\

Let us recall the definition of hyperbolicity of a graph in the sense of Gromov.

\begin{definition}

A connected graph $X$ is \emph{$\delta$-hyperbolic} for some $\delta >0$ if all geodesic triangles $[x,y,z]$ are $\delta$-thin, i.e. every side of the triangle is contained in the $\delta$-neighbourhood of the union of the other two sides.\\

A graph $X$ is said to be \emph{hyperbolic} if there exists $\delta>0$ such that $X$ is $\delta$-hyperbolic.     
\end{definition}

With this definition of $\delta$-hyperbolicity, according to \cite[Proposition 21]{GhysDelaHarpe}, each geodesic triangle admits a $4\delta$-quasi-center, i.e. for a geodesic triangle $[x,y,z]$ there exists a vertex $t\in X^{(0)}$ such that $d([x,y],t)\le 4\delta$, $d([x,z],t)\le 4\delta$ and $d([y,z],t)\le 4\delta$.\\

Let us recall the definition of the Gromov-product. For all $x,y,z \in X^{(0)}$, the Gromov-product between $x$ and $y$ seen from $z$ is defined by :

$$ (x,y)_{z}=\frac{1}{2}(d(x,z)+d(y,z)-d(x,y)). $$

In the rest of the paper, we will need a particular quasi-center for geodesic triangles defined in terms of Gromov product.

\begin{proposition} \label{quasicentreparticulier}

Let $x,y,z$ be vertices in $X$ and $[x,y,z]$ a geodesic triangle between these three points. We denote :

\begin{itemize}
    \item $u$ the vertex on $[x,z]$ such that $d(x,u)=(y,z)_{x}$ and $d(z,u)=(x,z)_{y}$,
    \item $v$ the vertex on $[y,z]$ such that $d(y,v)=(x,z)_{y}$ and $d(z,v)=(x,y)_{v}$,
    \item $w$ the vertex on $[x,y]$ such that $d(x,w)=(y,z)_{x}$ and $d(y,w)=(x,z)_{y}$.
\end{itemize}

There exists a vertex $t$ in $X$ such that, $d(u,t)\le 4\delta$, $d(v,t)\le 4\delta$ and $d(w,t) \le 4\delta$, i.e. $t$ is a $4\delta$-quasi center of $[x,y,z]$ such that $u$, $v$ and $w$ are close points on each side.
    
\end{proposition}

\begin{definition}{(Fine graph)} \cite[Proposition 2.1.]{Bowditchrelhyp}

 A graph is \emph{fine} if for every edge $e$, for all $L\ge0$, the set of simple simplicial loops of length at most $L$ which contain $e$ is finite. \\

 It is \emph{uniformly fine} if this set has cardinality bounded above by a function depending only on $L$. We denote by $\varphi$ this function of uniform finesse.

\end{definition}

This property is used by Bowditch in \cite[Theorem 7.10.]{Bowditchrelhyp} to give a combinatorial definition of relative hyperbolicity, equivalent to the bounded coset penetration condition (see \cite{Farbrelhyp}), as we explain below in Definition \ref{Definition Groupe relativement hyperbolique}.

\begin{ex}

\begin{itemize}
    \item Any locally finite graph is fine.
    
    \item Trees are fine, since they have no loops.

\end{itemize}

\end{ex}

\begin{definition}(Coned-off graph)

Let $G$ be a finitely generated group and $H_{1},...,H_{k}$ subgroups of $G$. We consider $Cay(G)$ a Cayley graph of $G$ with respect to a finite generating set. The coned-off of $G$ relatively to the subgroups $H_{1},...,H_{k}$ is the graph constructed as follows : for each $i$ and for each left coset $gH_{i}$ of $H_{i}$, add a vertex $\widehat{gH_{i}}$ and for each $h \in H_{i}$, add an edge starting at $\widehat{gH_{i}}$ and ending at $gh_{i}$. We denote this graph by $\widehat{Cay}(G)$.

\end{definition}

Lets us recall now Bowditch's definition \cite[Theorem 7.10.]{Bowditchrelhyp} of relative hyperbolicity.

 \begin{definition} \label{Definition Groupe relativement hyperbolique}
 
A pair $(G,{H_{1},...,H_{k}})$ of a finitely generated group $G$ and a collection of subgroups is relatively hyperbolic, if one (or equivalently every) coned-off Cayley graph $\widehat{Cay}(G)$ over $H_{1},...,H_{k}$ is hyperbolic and uniformly fine.

 \end{definition}

 \begin{ex}
 
 \begin{itemize}
     
     \item If $K$ and $N$ are two finitely generated groups, then $\Gamma=K\star N$ is hyperbolic relatively to $(K,N)$.
     
     \item Let $M$ be a complete Riemannian manifold of finite volume, with pinched negative sectionnal curvature, and with a single cusp. Then according to \cite[Theorem 4.11]{Farbrelhyp}
     , the fundamental group of $M$, $\pi_{1}(M)$ is hyperbolic relatively to the cusp subgroup.

 \end{itemize}

 \end{ex}

\begin{nonex}

According to \cite[Theorem 1.3]{MasurMinsky} , the coned-off of the Mapping class group of an oriented closed surface relatively to stabilizers of simple closed curves is hyperbolic. Nevertheless, this graph is not fine therefore Mapping class groups are not relatively hyperbolic with respect to curve stablilizers.

\end{nonex}
 One difficulty of working on relatively hyperbolic groups comes from the fact that coned-off graphs are not locally finite. Angles and cones, which we will define bellow, allow to solve this difficulty. 

\subsection{Angles and cones}
 
Following Chatterji and Dahmani in \cite{ChatterjiDahmani}, let us recall the definition of angle .

\begin{definition} {(Angle)}

Let $X$ be a graph. Given a vertex $v \in X^{(0)}$ and two unoriented edges $e_{1}=\{v,w\}$ and $e_{2}=\{v,w'\}$ such that  $v \in e_{1}\cap e_{2}$, the \emph{angle} between $e_{2}$ and $e_{1}$ at $v$, denoted by $\measuredangle_{v}(w,w')$, is defined as follows :

$$ \measuredangle_{v}(e_{1},e_{2})=d_{X   \backslash \{ v \} }(w,w').$$

In other words, the \emph{angle} between $e_{1}=\{v,w\}$ and $e_{2}=\{v,w'\}$ at $v$ is the infimum of length of paths from $w$ to $w'$ that avoid $v$.
In particular, an angle could be equal to $+\infty$.

By convention, we will use the following abuse of notation : if $x_{1}, x_{2}$ and $v$ are vertices of $X$, distinct from $v$, we say that $ \measuredangle_{v}(x_{1},x_{2}) > \theta$ if there exist two edges $e_{1}$, $e_{2}$, such that $ v \in e_{1}\cap e_{2} $,  with $e_{i}$ on a geodesic between $v$ and $x_{i}$ $(i=1,2)$ and such that $ \measuredangle_{v}(e_{1},e_{2}) > \theta$. 
We say that $ \measuredangle_{v}(x_{1},x_{2}) \leq \theta$ otherwise, i.e., for all geodesics between $v$ and $x_{i}$ $(i=1,2)$ and starting by the edges $e_{i}$ $(i=1,2)$ such that $v \in e_{1}\cap e_{2}$, then $\measuredangle_{v}(e_{1},e_{2})\leq \theta$.

\end{definition}

 \begin{proposition} \label{Triangle inequality} (Triangle inequality)\\ For all edges $e_{1}$, $e_{2}$ and $e_{3}$ and vertex $v$ such that $ v \in e_{1}\cap e_{2} \cap e_{3}$ , we have :

 $$ \measuredangle_{v}(e_{1},e_{3})\le \measuredangle_{v}(e_{1},e_{2}) + \measuredangle_{v}(e_{2},e_{3}) .$$

 Let $a,b,c \in X^{(0)} \setminus \{v\}$, we have also :
 
$$ \measuredangle_{v}(a,b)\le \measuredangle_{v}(a,c) + \measuredangle_{v}(c,b) .$$

 \end{proposition}

 \begin{proof}
     The first property comes directly from the triangle inequality in $X \setminus \{v\}$.\\

     To deduce the second inequality with vertices, we just have to come back to the definition of angle and use the first point.

 \end{proof}

In a $\delta$-hyperbolic graph, large angles at a vertex force geodesics to pass through this vertex.

\begin{proposition}\label{angletoutegéod}\cite[Proposition 1.2]{ChatterjiDahmani} 
Let $X$ be a $\delta$-hyperbolic graph and $a,b,c$ vertices in $X$, $a,b \neq c$, if $\measuredangle _{c}(a,b) > 12\delta$, then every geodesic between $a$ and $b$ goes through $c$.
\end{proposition}

Following Chatterji and Dahmani in \cite{ChatterjiDahmani}, let us define cones.

\begin{definition} {(Cone)}\\
Let $X$ be a graph. Let $e$ be an edge
and a number $\theta>0$. \\

The cone of parameter $\theta$ around the edge, denoted by $Cone_{\theta}(e)$, is the subset of vertices and edges of $X$, $v$ and $e'$ such that there is path from length at most $d$ from $e$ to $e'$ and for which two consecutive edges make an angle at most $\theta$, i.e. :

$$\begin{aligned} 
         Cone_{\theta}(e) & =\{e' ~|~ \exists e_{0}=e, e_{1}, ..., e_{n}=e' \text{such that}~ n\le \theta,  \forall i , \exists v_{i}\in e_{i} \cap e_{i+1} \text{and} \measuredangle_{v_{i}}(e_{i},e_{i+1}) \le \theta \} \\
         & \cup \{ v ~|~ \exists e'=\{v,w\} ~\text{and}~ e_{0}=e, e_{1}, ..., e_{n}=e' \text{such that}~ n\le \theta,  \forall i , \exists v_{i}\in e_{i} \cap e_{i+1} \text{and} \measuredangle_{v_{i}}(e_{i},e_{i+1}) \le \theta \} .
\end{aligned}$$

\end{definition}

\begin{proposition} \label{cardinal cones}\cite[Propostion 1.10]{ChatterjiDahmani} 
In a fine graph, cones are finite.\\ 

If the graph is uniformly fine, for all $e\in  X^{(1)} $ and $\theta > 0$, the cardinality of $Cone_{\theta}(e)$ can be bounded above by a function of $\theta$.

\end{proposition}  

The following property allows to use cones in relatively hyperbolic groups as a sort of finite neighbourhoods of geodesics.

\begin{proposition} \cite[Propostion 1.10]{ChatterjiDahmani} \label{conetriangle} {(Cones and triangles)} 

In a $\delta$-hyperbolic graph, geodesic triangles are conically fine, i.e. for any geodesic triangle $[a,b,c]$, every edge $e$ on $[a,b]$ is contained in a cone of parameter $50\delta$ around an edge $e'$ that is either on $[a,c]$ at the same distance of $a$ than $e$, either on $[b,c]$ at the same distance of $b$ than $e$. 
\end{proposition}

A corollary of this proposition is the fact that in a uniformly fine graph intervals are uniformly finite even if the graph is not locally finite.

\begin{corollary}

In a uniformly fine and $\delta$-hyperbolic graph $X$, for all $x,x' \in X$,  the cardinality of $I(x,x')$ is bounded above by a function depending of $d(x,x')$, of $\delta$ and of the function of uniform finesse $f$.

\end{corollary}

The following lemma of \cite[Lemma 1.15]{BridsonHaefliger} will be useful for the next result.

\begin{lem} \cite[Lemma 1.15]{BridsonHaefliger} \label{lemme de Bridson} 

Let $X$ be a geodesic space such that geodesic triangles are $\delta$-thin. Let $\gamma,\gamma' : [0,T] \mapsto X$ be geodesics with $\gamma(0)=\gamma'(0)$. If $d(\gamma(t_{0}),im(\gamma'))\le K$, for some $K>0$ and $t_{0} \in [0,T]$, then $d(\gamma(t),\gamma'(t))\le 2\delta$ for all $t \le t_{0}-K-\delta$.

\end{lem}

\begin{proposition}\label{conetriangleplusfort}

For every geodesic triangle $[a,b,c]$, if $e$ is an edge on $[a,b]$ such that $d(a,e)\le (b,c)_{a}-8\delta$ then if $e'$ is the edge on $[a,c]$ at the same distance of $a$ than $e$, $e$ is contained in the cone of parameter $50\delta$ around $e'$ and conversely.
    
\end{proposition}

\begin{proof}

Let $[a,b,c]$ be geodesic triangle between $a,b$ and $c$, let $\gamma$ be the geodesic $[a,b]$ from $a$ to $b$ and $\gamma'$ be the geodesic $[a,c]$ from $a$ to $c$.\\

Be definition of $\gamma$ and $\gamma'$, we have that $\gamma(0)=\gamma'(0)=a$. Moreover, thanks to hyperbolicity, we know that $d(\gamma((b,c)_{a}),\gamma'((b,c)_{a})\le 4\delta$. Then we can apply Lemma \ref{lemme de Bridson}, to the geodesics $\gamma$, $\gamma'$ with $K=4\delta$ and $t_{0}=(b,c)_{a}$.

Let us denote by $\sigma_{1}$ the subsegment of $[a,b]$ obtained from $e$ and extended by $3\delta$ from both sides, unless it reaches $a$ before. Let us denote $\sigma_{1}^{0}$ its initial point and $\sigma_{1}^{1}$ its endpoint.

 By assumption $d(a,e)\le (b,c)_{a}-8\delta$, therefore   $d(a,\sigma_{1}^{1})\le (b,c)_{a}-5\delta$. According to Lemma \ref{lemme de Bridson}, there exist two segments at most $2\delta$ long, $\sigma_{0}$ from $\sigma_{1}^{0}$ to $[a,c]$  and $\sigma_{2}$ from $\sigma_{1}^{1}$ to $[a,c]$. Let us denote $\sigma_{3}$, the subsegment of $[a,c]$ that closes the loop $\sigma_{0}^{-1}\sigma_{1}\sigma_{2}\sigma_{3}$. By the triangular inequlatity, $\sigma_{3}$ is of length at most $11\delta$. Furthermore by the triangular inequality, $e$ cannot be in the transition arcs $\sigma_{0}$, $\sigma_{2}$. If $e \in \sigma_{3}$, the result is clear. Otherwise, the loop $\sigma_{0}^{-1}\sigma_{1}\sigma_{2}\sigma_{3}$ goes only one time through $e$ and it is at most $22\delta$ long , then if $e'$ is the edge of $[a,c]$ at the same distance of $a$ than $e$, we have that $e \in Cone_{50\delta}(e')$.

\end{proof}

With this new version of the fact that geodesic triangles are conically thin, we know that at a uniform distance of the center of a geodesic triangle, every edge of a side belongs to a cone of a bounded parameter of an edge of the closest side of the triangle. \\

\subsection{Triangles in fine graphs}

In this paragraph, we will prove the main result of this section, which is the existence of a particular geodesic triangle between any three points in a fine and hyperbolic graph. Throughout this section, $X$ denotes a fine hyperbolic graph.

\begin{proposition}
    
\label{pointloingrandangleexistence}
Let $a,b,c \in X^{(0)}$, we consider the set of points $v$ with the following properties :

\begin{itemize}
    \item $v$ is contained in every geodesic between $a$ and $b$ and between $a$ and $c$,
    \item $\measuredangle_{v}(a,b)> 50\delta$,
    \item $\measuredangle_{v}(a,c)> 50\delta$.
\end{itemize}

If this set is not empty, it contains a unique point at maximal distance of $a$.

\end{proposition}

\begin{proof}

Let $v$ and $v'$ be two vertices following the properties of the Proposition at maximal distance of $a$. In particular, we know that $d(a,v)=d(a,v')$.

Let $\gamma$ be a geodesic from $a$ to $b$, thank to the first property, we have that $v \in \gamma$ and $v' \in \gamma$, and we the fact that $d(a,v)=d(a,v')$, one can deduce that $v=v'$.

\end{proof}

\begin{definition} \label{pointloingrandangle}
Let $a,b,c \in X^{(0)}$, according to Proposition \ref{pointloingrandangleexistence}, we define $\tilde{a}$ as the furthest point from $a$ among the points $v$ with the following properties :

\begin{itemize}
    \item $v$ is contained in every geodesic between $a$ and $b$ and between $a$ et $c$,
    \item $\measuredangle_{v}(a,b)> 50\delta$,
    \item $\measuredangle_{v}(a,c)> 50\delta$.
\end{itemize}

If the set of points following these three properties is empty, we set $\tilde{a}=a$.\\

We define $\tilde{b}$ and $\tilde{c}$ in the same way.

\end{definition}

\begin{proposition}

Let $a,b,c \in X^{(0)}$, we have $d(a,\tilde{a})\leq d(a,\tilde{c})$. In other words, when we travel along a geodesic from $a$ to $c$, we pass through $a$, $\tilde{a}$, $\tilde{c}$ and $c$ in this order.

\end{proposition}

\begin{proof}

By contradiction, assume that $d(a,\tilde{c})<d(a,\tilde{a})$.\\

By definition, the vertices $\tilde{a}$ and $\tilde{c}$ belong to every geodesic between $a$ and $c$. By assumption, when we travel along a geodesic from $a$ to $c$, we pass through the vertices $a$, $\tilde{c}$, $\tilde{a}$ and $c$ in this order. In particular, $\tilde{c}$ belongs to every geodesic between $a$ and $\tilde{a}$.\\

Moreover by definition, $\tilde{a}$ belongs to every geodesic between $a$ et $b$, so $\tilde{c}$ belongs to every geodesic between $a$ et $b$. We can deduce that $\tilde{a}$ belongs to some geodesic between $\tilde{c}$ and $b$. In particular, we have $d(\tilde{c},b)=d(\tilde{c},\tilde{a})+d(\tilde{a},b)$ and then :\\

\begin{center}
$d(\tilde{a},b)<d(\tilde{c},b)$, see Figure \ref{ atilde est plus prÃ¨s de a que ctilde_1}.
    
\end{center}

\begin{figure}[!ht]
    \centering
    \includegraphics[scale=0.7]{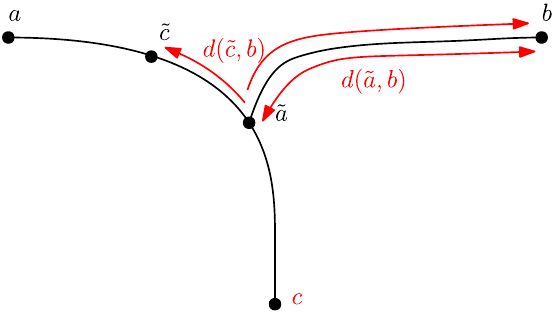}
    \caption{ $\tilde{a}$ is closer to $a$ than $\tilde{c}$ }
    \label{ atilde est plus prÃ¨s de a que ctilde_1}
\end{figure}

The same argument switching the roles of $a$ and $c$ and those of $\tilde{a}$ and $\tilde{c}$ gives the reverse inequality $ d(\tilde{c},b)<d(\tilde{a},b) $. This concludes the proof by contradiction. Therefore we have $d(a,\tilde{a})\leq d(a,\tilde{c})$.

\end{proof}

\begin{theo} \label{formenormaledestriangles}

Let $a,b,c \in X^{(0)}$, there exists a geodesic triangle $[a,b,c]$ between these three vertices such that :
\begin{itemize}
    \item the sides $[a,b]$ and $[a,c]$ coincide between $a$ and $\tilde{a}$,
    \item the sides $[b,c]$ and $[b,a]$ coincide between $b$ and $\tilde{b}$,
    \item the sides $[c,a]$ and $[c,b]$ coincide between $c$ and $\tilde{c}$,
    \item for each $v \in ]\tilde{a},\tilde{b}[$, if we denote by $e_{1}$ and $e_{2}$ the two different edges of $[\tilde{a},\tilde{b}]$ such that $v \in e_{1} \cap e_{2}$ then $\measuredangle_{v}(e_{1},e_{2}) \leq 100\delta$ and similarly this property is also true for $]\tilde{a},\tilde{c}[$ and $]\tilde{b},\tilde{c}[$.\\

\end{itemize}

\end{theo}

\begin{proof} 

By definition $\tilde{a}$ belongs to every geodesic between $a$ and $b$, and to every geodesic between $a$ and $c$, thus we can choose a geodesic from $a$ to $\tilde{a}$, denoted $ [ a , \tilde{a}] $, and extend it to a geodesic between $a$ and $b$, denoted $[a,b]$, and to a geodesic between $a$ and $c$, denoted $[a,c]$.\\

By definition of $\tilde{b}$ and $\tilde{c}$, we know that $\tilde{b} \in [a,b]$ and $\tilde{c} \in [a,c]$. Thus, there exists a geodesic between $b$ and $\tilde{b}$, denoted $[b,\tilde{b}]$, such that $[a,b]=[a,\tilde{a}]\cup [\tilde{a},\tilde{b}] \cup [\tilde{b},b]$. In the same way, there exists a geodesic between $c$ and $\tilde{c}$, denoted $[c,\tilde{c}]$ such that $[a,c]=[a,\tilde{a}]\cup [\tilde{a},\tilde{c}] \cup [\tilde{c},c]$. \\

For the side $[b,c]$, we choose a geodesic between $\tilde{b}$ and $\tilde{c}$, and we take $[b,c]=[ b,\tilde{b}]\cup [\tilde{b},\tilde{c}] \cup [\tilde{c},c]$.\\

It remains to show that the triangle satisfies the last point of the theorem. By contradiction, we assume that there exists a vertex $v\in ]\tilde{a},\tilde{b}[$ such that $\measuredangle_{v}(e_{1},e_{2}) > 100\delta $ where $e_{1}$ and $e_{2}$ are the two edges on the triangle such that $v \in e_{1} \cap e_{2}$ see Figure \ref{ fig: forme normale triangle}. In particular, we have that $\measuredangle_{v}(a,b) > 100\delta $ and according to Proposition \ref{angletoutegéod}, $v$ belongs to every geodesic between $a$ and $b$. If we use Proposition \ref{Triangle inequality}, we can remark that :

$$  \measuredangle_{v}(a,b) \le  \measuredangle_{v}(a,c) +  \measuredangle_{v}(c,b) . $$

Thus $\measuredangle_{v}(a,c) > 50\delta $ or $ \measuredangle_{v}(c,b) > 50\delta$. If the first inequality is true, according to Proposition \ref{angletoutegéod}, we deduce that $v$ belongs to every geodesic between $a$ and $b$ and between $a$ and $c$. This leads to a contradiction with Definition \ref{pointloingrandangle} of $\tilde{a}$, in fact $v$ is a vertex verifying all the points of the definition such that $d(a,v)>d(a,\tilde{a})$.\\

If the second inequality is true, we obtain the same contradiction for $\tilde{b}$. This proves the theorem.
\end{proof}

\begin{figure}[!ht]
    \centering
    \includegraphics[scale=0.7]{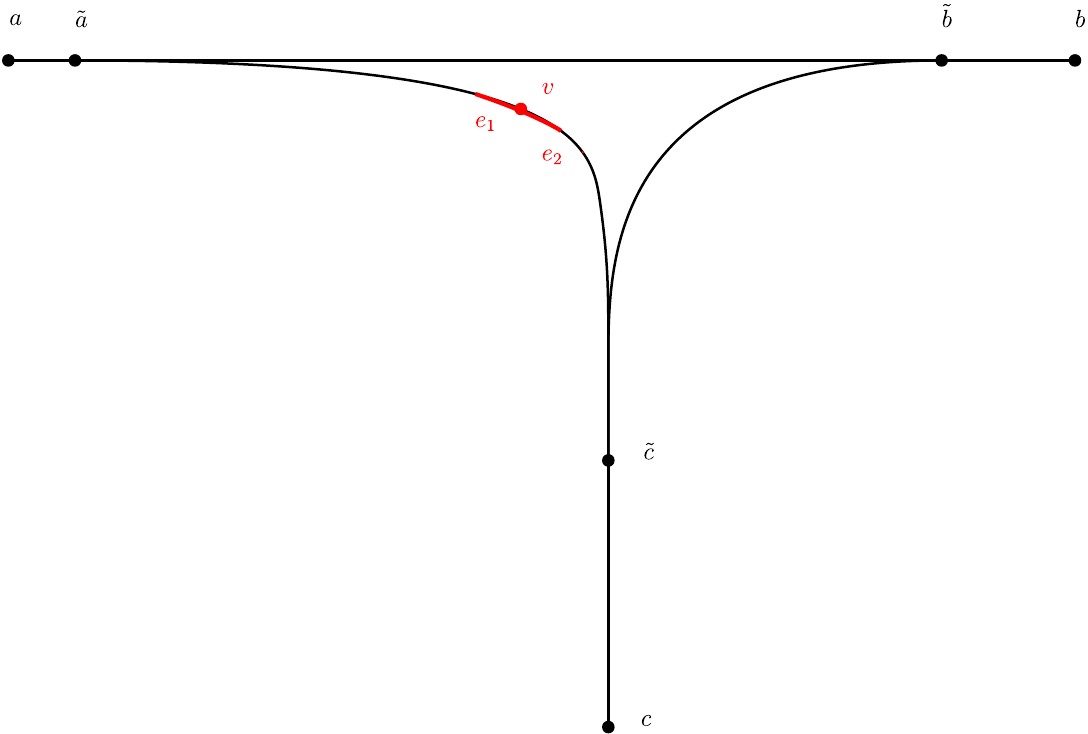}
    \caption{ Particular geodesic triangle between the three points $a$, $b$, $c$ with $\measuredangle_{v}(e_{1},e_{2}) \leq 100\delta$ }
    \label{ fig: forme normale triangle}
\end{figure}

\newpage

\subsection{Proof of Corollary \ref{corollary : propriété de rigidité sur morphisme vers groupe rel hyp}}

This section is dedicated to the proof of Corollary \ref{corollary : propriété de rigidité sur morphisme vers groupe rel hyp}.\\ Let $G$ denote a group with the strong Property $(T)$. Let $\Gamma$ be a group hyperbolic relative to $(H_{1},...,H_{k})$, where $H_{i}$ for $1 \le i \le k$ are subgroups of $\Gamma$. Let $X$ denote a coned-off graph of $\Gamma$.

$G$ acts isometrically on $X$ via $\varphi$. According to Theorem \ref{theorem:main result}, this action should have bounded orbits or equivalently $\varphi(G)$ acts on $X$ with bounded orbits. Thus, Corollary \ref{corollary : propriété de rigidité sur morphisme vers groupe rel hyp} is equivalent to the following Proposition.

\begin{proposition}\label{proposition:sous-groupe d'un groupe relativement hyperbolique avec des orbites bornées }

Let $\Gamma$ be group relative to $(H_{1},...,H_{k})$, $G$ a subgroup of $\Gamma$.
If $G$ acts on a coned-off graph of $\Gamma$ with bounded orbits, then $G$ is either finite or contained in a conjugate of a $H_{i}$.
    
\end{proposition}

This question seems to be known by specialists, nevertheless we will give a proof of this fact for the reader convenience.

In the rest of the section, $\Gamma$ denotes a group hyperbolic relative to $(H_{1},...,H_{k})$, $X$ a coned-off graph of $\Gamma$ and $G$ a subgroup of $\Gamma$, which acts on $X$ with bounded orbits. $d$ will denote the graph metric of $X$. Let us assume that $G$ is infinite, we will prove that $G$ stabilizes a parabolic point in $X$.

\begin{definition}\label{definition : fonction de déplacement maximale }

Let $x$ be a vertex of $X$, we define for all $n \in \mathbb{N}$ :
$$ \mathcal{A}_{n}(x)=\{g \in G ~|~ d(x,g.x)=n \}.$$

We set : 
$$f(x)=\max \{ n \in \mathbb{N} ~|~ |\mathcal{A}_{n}(x)| > +\infty \}. $$

\end{definition}

Let $x_{0}$ a vertex in $X$, such that $f(x_{0})=K$ is the minimum of $f$. We will show that $K=0$. Let us assume by contradiction that $K>0$.

\subsubsection{The case where $x_{0} \in Cay(G)$ }

Let us assume that $x_{0} \in Cay(\Gamma)$.

\begin{lem}\label{lemma : il existe un point parabolique de grand angle}

There exists $g \in G$ such that :

\begin{itemize}

\item $d(x_{0},g.x_{0})=K$,

\item and there exists a vertex $p$ different from $x$ such that $\measuredangle_{p}(x_{0},g.x_{0})> 24\delta$.

\end{itemize}

\end{lem}

\begin{proof}
To prove the existence of a such element $g$ in the group, we will use the distance formula in relatively hyperbolic groups as it is stated in \cite[Proposition 1.6.]{ChatterjiDahmani}. Let $d_{w}$ be the word metric associated to a generating set of $G$ containing the generating set of parabolic subgroups.

There exist $A,B>1$ for which, for all $g \in G$, for each choice of geodesic $[1,g]$ in $X$, if $\sum \measuredangle([1,g])$ denotes the sum of angles of edges of $[1,g]$ at vertices of infinite valence, we have :

$$ \frac{1}{A}d_{w}(1,g)-B \leq d(1,) + \sum \measuredangle([1,g]) \leq Ad_{w}(1,g)+B . $$

We will just use the first inequality applied to $x_{0}$ and $g.x_{0}$. Since $f(x_{0})=K$, there are infinitely many $g \in G$ such that $d(x_{0},g.x_{0})=K$. In particular, we can find $g_{n} \in G$, such that $d(x_{0},g_{n}.x_{0})=K$ and $d_{w}(x_{0},g_{n}.x_{0}) \rightarrow \infty$ as $n \rightarrow \infty$. Therefore, thank to the distance formula, we deduce that $ \sum \measuredangle([x_{0},g_{n}.x_{0}]) \rightarrow \infty$.\\

Since the orbits of $G$ on $X$ are bounded, a geodesic between $x_{0}$ and $g_{n}.x_{0}$ should meet a finite number of parabolic points. Thus, there exists $g_{n}$ such that $d(x_{0},g_{n}.x_{0})=K$ and a parabolic point $p$ such that $\measuredangle_{p}(x_{0},g_{n}.x_{0})$ is arbitrarily large. Moreover since $x_{0} \in Cay(\Gamma)$, $p$ is different from $x_{0}$.

\end{proof}

Let us consider $g\in G$ and a parabolic vertex $p$ as in \ref{lemma : il existe un point parabolique de grand angle}. Since $f(x_{0})=K$ is minimal, $f(p) \ge K$. Therefore, there exists infinitely many $h \in G$ such that $d(h,h.p) \ge K$.\\

According to Proposition \ref{Triangle inequality}, we have that :

\begin{itemize}
    \item $ \measuredangle_{p}(x_{0},g.x_{0}) \leq \measuredangle_{p}(x_{0},h.p) + \measuredangle_{p}(h.p,g.x_{0}) $,

    \item $ \measuredangle_{h.p}(h.x_{0},hg.x_{0}) \leq \measuredangle_{h.p}(h.x_{0},p) + \measuredangle_{h.p}(h.p,hg.x_{0}) $.
\end{itemize}

Therefore, for each of the two inequalities, for infinitely many $h \in G$ one of the two angles on the right side of the inequality should be greater than $12\delta$. This leads to four cases.\\

Let us assume that we are in the case where $\measuredangle_{p}(x_{0},h.p) > 12 \delta$ and $\measuredangle_{h.p}(h.x_{0},p) > 12 \delta$ for infinitely many $h \in G$. According to Proposition \ref{angletoutegéod}, this implies that :
\begin{itemize}
    \item $p \in I(x_{0},h.p)$,
    \item and $h.p \in I(h.x_{0},p)$.
\end{itemize}

With fact, we deduce that $\measuredangle_{p}(x_{0},h.x_{0}) > 12 \delta$ and then :

$$\begin{aligned} 
           d(x_{0},h.x_{0})&=d(x_{0},p)+d(p,h.x_{0})~(\text{since } p \in I(x_{0},h.x_{0})  \\
            & =d(x_{0},p)+d(p,h.p)+d(h.p,h.x_{0})~(\text{since } h.p \in I(h.x_{0},p))\\      
            & > K. ~(\text{since } d(x_{0},p)>0).
\end{aligned}$$

To conclude, we get infinitely many $h$ such that $d(x_{0},h.x_{0})>K$. This is a contradiction of the fact that $f(x_{0})=K$. The same proof works for the three other cases. Then $x_{0}$ should be a parabolic point.

\subsubsection{The case where $x_{0}$ is a parabolic point}

Let us assume that $x_{0}$
is a parabolic point. If there exists $g$ such that $d(x_{0},g.x_{0})=K$ and such that there exists another vertex $p$ with $\measuredangle_{p}(x_{0},g.x_{0})> 24\delta$, the proof of the previous paragraph applies also here.\\

Let us assume that there are no $g \in G$ and no vertex $p$ such that $\measuredangle_{p}(x_{0},g.x_{0})>24\delta$. In this case, we deduce that :

$$G.x_{0}\cap S(x_{0},K) \subset Cone_{\max(K,24 \delta)}(p_{0}).$$

According to \ref{cardinal cones}, we deduce that $G.x_{0}\cap S(x_{0},K)$ is finite. \\

Since there are infinitely many $g$ such that $d(x_{0},g.x_{0})=K$. We deduce that there exists a vertex in $G.x_{0}$ with an infinite stabilizer in $G$. By conjugation, we deduce that $G_{x_{0}}:=G\cap Stab(x_{0})$ is infinite.\\

Let us a consider $g \in G$ such that $d(x_{0},g.x_{0})=K$ and $x \in I(x_{0},g.x_{0})$ such that $d(x,x_{0})=1$. Since the action of $G$ on $Cay(\Gamma)$ is free and $G_{x_{0}}$ is an infinite subgroup of $\Gamma$, we have that $G_{x_{0}}.x$ is infinite. \\

Therefore, we can find $\gamma \in G_{x_{0}}$ such that $\measuredangle_{x_{0}}(x,\gamma.x)> 12 \delta$. Otherwise, if every angle was small, we could include $G_{x_{0}}.x$ in a cone centered at $x_{0}$, which contradicts the fact that it is infinite.\\

The fact that $\measuredangle_{x_{0}}(x,\gamma.x)> 12 \delta$ implies that $\measuredangle_{x_{0}}(g.x_{0},\gamma g.x_{0})> 12 \delta$. Thus $x_{0} \in I(g.x_{0},\gamma g.x_{0})$, then : $$
d(g.x_{0},\gamma g.x_{0})=d(g.x_{0},x_{0})+d(x_{0},\gamma g.x_{0})=2K .$$ 

Then $d(x_{0},g^{-1}\gamma g.x_{0})=2K$, using the fact that every vertex stabilizer is infinite, we obtain infinitely many $h \in G$ such that $d(x_{0},h.x_{0})=2K$, which contradicts the fact that $f(x_{0})=K$.\\

To conclude, the contradiction implies that the minimum of $f$ is $0$. Once again, the action of $G$ on $Cay(\Gamma)$ is free, thus the minimum of $f$ is reached on a parabolic point $x_{0}$. We will prove that this point is stabilizes by $G$. Since $f(x_{0})=0$, $G_{x_{0}}$, the stabilizer of $x_{0}$ in $G$ is infinite. If there exists $g \in G$ such that $g.x_{0}\neq x_{0}$, by conjugation the stabilizer of $g.x_{0}$ should also be infinite, which contradicts the fact that $f(x_{0})=0$. Therefore $G$ stabilizes a parabolic vertex and this prove Corollary \ref{corollary : propriété de rigidité sur morphisme vers groupe rel hyp}.

\newpage

\section{The case of a locally infinite tree}\label{section : le cas des arbres}

When $X$ is a tree the result of Theorem \ref{theorem:main result}  is already known and it is also true with Property $(T)$ or even with Property FA only. Nevertheless, trees are $0$-hyperbolic and they are particular examples of fine graphs.

We will prove the theorem in this case to illustrate the strategy of proof in the simplest setting. In Lafforgue's article \cite{Lafforgue}, the case of a uniformly locally finite tree is proved, here we will make no further assumption on the tree, thus difficulties will come from the fact that $X$ could be locally infinite.\\

Let us assume that $G$ is a locally compact group having strong Property $(T)$ and let us assume that $X$ is a tree equipped with an isometric action of $G$. We will prove the following theorem :

\begin{theo} \label{theoremeimportant arbre}
Let $G$ be a locally compact group having strong Property $(T)$. Let $X$ be  a tree equipped with an isometric action of $G$. Then every orbit for this action is bounded.

\end{theo}

\subsection{Angles and cones in trees}

For all $x,y \in X$, we denote by $[x,y]$ the unique geodesic between $x$ and $y$. We will start to describe the concepts of the previous section in this particular setting.

\begin{proposition} \label{arbres angles and cones} (Angles and cones in a tree)\\

Let $e_{1}$ and $e_{2}$ be edges of the tree $X$ such that $v \in e_{1} \cap e_{2}$ then :

\begin{itemize}

    \item $\measuredangle_{v}(e_{1},e_{2})=\infty$ if and only if $e_{1} \neq e_{2}$,

     \item $\measuredangle_{v}(e_{1},e_{2})=0$ if and only if $e_{1}=e_{2}$.

\end{itemize}

For all vertices $a$, $b$, $v$ such that $a,b \neq v$ of the tree $X$ :

\begin{itemize}
    \item $\measuredangle_{v}(a,b)=\infty$ if and only if $v$ separates $a$ and $b$,

    \item $\measuredangle_{v}(a,b)=0$ otherwise.
    
\end{itemize}

Therefore, for each edge $e=\{x,y\}$ of $X$ and for all $\theta>0$ :

$$ Cone_{\theta}(e)= \{x,e, y \} .$$

\end{proposition}

\begin{proof}

To prove the first point, we remark that since $v \in e_{1} \cap e_{2}$, the unique path between the other vertex of $e_{1}$ and of $e_{2}$ goes trough $v$ if and only if they are different. We deduce the first point by definition of angles.

The second point is clear with the definition of angles for vertex.

The third point is also clear by definition of cones and with the fact that the angle between different edges is infinite.

\end{proof}

We will now describe the triangle of Theorem \ref{formenormaledestriangles} in the case of a tree. To be more precise, between every three points, there is a unique geodesic triangle in a tree, which a tripod. We will show that this triangle has the properties of Theorem \ref{formenormaledestriangles}.

\begin{proposition}
Let $a,b,c \in X$. We denote $[a,b,c]$ the unique geodesic triangle between these three points and $t$ its center. Then : $$\tilde{a}=\tilde{b}=\tilde{c}=t $$
where $\tilde{a}, \tilde{b}, \tilde{c}$ are the points of Definition \ref{pointloingrandangle}.

Moreover, $[a,b,c]$ is a triangle satisfying the properties of Theorem \ref{formenormaledestriangles}.
    
\end{proposition}

\begin{proof}

By Definition \ref{pointloingrandangle}, $\tilde{a}$ is the furthest point from $a$ on every geodesic between $a$ and $b$ and every geodesic between $a$ and $c$ such that every angle between $a$ and $\tilde{a}$ is greater than $50\delta$. Therefore according to Proposition \ref{arbres angles and cones}, $\tilde{a}=t$ and the same argument works for $\tilde{b}$ and $\tilde{c}$.

To check that $[a,b,c]$ is a triangle like the one given by Theorem \ref{formenormaledestriangles}, we do not have to check the fourth point since $\tilde{a}=\tilde{b}=\tilde{c}$, the three other points are direct since trees are uniquely geodesic.
    
\end{proof}

\subsection{Strategy of the proof}

We are now ready to prove Theorem \ref{theoremeimportant arbre} and to see that this will be a particular case of the proof in the general case.\\

For all $n \in \mathbb{N}$ and for all  $x \in X$, we denote by $S_{x}^{n}$ the sphere centered at $x$ and of radius $n$. For all $n, k \in \mathbb{N} $, $k\le n$, for all $z \in S_{x}^{k}$, we set $I_{z}^{n,k,x}=\{a\in S_{x}^{n}, z \in [x,a] \}$ following Lafforgue's notation in \cite{Lafforgue} such that $ S_{x}^{n}=\bigcup_{z\in S_{x}^{k}} I_{z}^{n,k,x} $ defines a partition of the sphere $ S_{x}^{n}$. For $k=n$, the partition is the partition with singletons and for $k=0$, this is the trivial partition. \\

We denote by $\mathbb{C}^{(X)}$ the vector space of functions from $X$ to $\mathbb{C}$ with finite support with canonical basis $(e_{a})_{a\in X}$. We will describe the Hilbert space used by Lafforgue in his article \cite{Lafforgue}.

\begin{proposition} \label{c'est une norme}

For each $x \in X$, the following expression :
$$ \left\| \sum_{a} f(a)e_{a} \right\|_{H_{x}}^{2}= \sum_{n \in \mathbb{N}} (n+1)^{2} \sum_{k=0}^{n} \sum_{z \in S_{x}^{k} } { \biggr\vert \sum_{a \in I_{z}^{n,k,x}} f(a) \biggr\lvert }^{2}$$ defines the square of a norm on $\mathbb{C}^{(X)}$ with canonical basis $(e_{a})_{a\in X}$.
    
\end{proposition}

\begin{proof}

For every $f \in \mathbb{C}^{(X)} $, since the support of $f$ is finite, for $n$ big enough, if $a \in I_{z}^{n,k,x} $, we have that $f(a)=0$. Then, for $n$ big enough, we have that :  $$\sum_{k=0}^{n} \sum_{z \in S_{x}^{k} } { \biggr\vert \sum_{a \in I_{z}^{n,k,x}} f(a) \biggr\lvert }^{2} =0 .$$

Thus, the sum in the definition of the norm of $f$ is a finite and therefore the norm is well defined.\\

With the definition of the norm, the triangle inequality and the absolute homogeneity are clear. Let us look at the positive definiteness.\\

Let $f \in \mathbb{C}^{(X)}$ such that $\left\| \sum_{a} f(a)e_{a} \right\|_{H_{x}} =0 $. With the fact that the partition of the sphere $S_{x}^{n}$ associated to $k=n$ is the partition with singletons, we have that :

$$\begin{aligned} 
           \left\| \sum_{a} f(a)e_{a} \right\|_{H_{x}}^{2} & = \sum_{n \in \mathbb{N}} (n+1)^{2} \sum_{k=0}^{n} \sum_{z \in S_{x}^{k} } { \biggr\vert \sum_{a \in I_{z}^{n,k,x}} f(a) \biggr\lvert }^{2}\\      
            & \geq \sum_{n \in \mathbb{N}} (n+1)^{2} \sum_{z \in S_{x}^{n} } { \biggr\vert \sum_{a \in I_{z}^{n,n,x}} f(a) \biggr\lvert }^{2}\\
            & \geq \sum_{n \in \mathbb{N}} (n+1)^{2} \sum_{z \in S_{x}^{n} } { \vert f(z) \lvert }^{2}.
\end{aligned}$$

The last expression is zero if and only if $f(z)=0$ for all $z \in X$, therefore if and only if $f=0$. This prove the positive definiteness of the norm.

\end{proof}

According to Proposition \ref{c'est une norme}, we can define the Hilbert space $H_{x}$, which is the completion of $\mathbb{C}^{(X)}$ for the previous norm. In other words, $H_{x}$ is the set of functions $f$ from $X$ to $\mathbb{C}$ such that $\left\|  f \right\|_{H_{x}} $ is finite. \\

We will show that the action of $G$ by left translations on $\mathbb{C}^{(X)}$, given by the formula $\pi(g)(f)(x)=f(g^{-1}x)$ for $g \in G$, $f \in \mathbb{C}^{(X)}$, $x \in X$, extends to a continuous representation of $G$ on $H=H_{x}$ for some $x \in X$. Furthermore  we will prove that there exists a polynomial $P$ such that $||\pi(g)||_{\mathcal{L}(H)} \leq P(l(g))$ for all $g \in G$, that the dual representation $H^{\ast}$ admits a non-zero fixed vector, and that $H$ has a non-zero fixed vector if and only if orbits of the action of $G$ on $X$ are bounded.\\

\subsection{Comparison between equivalence classes}

For all $x \in X$, we set $U_{x} = \{(n,k,z), 0 \leq  k\leq n, z\in S_{x}^{n} \}$. For each $x \in X$, we define a length on $G$, $l_{x}$, by $l_{x}(g)=d(x,g.x)$ for all $g \in G$.

In the rest of the subsection, we fix $x \in X$ and we will denote by $l$ the length $l_{x}$.
Let $g \in G$, we set $x'=g(x)$ such that $d(x,x')=l(g)$. By definition of $H=H_{x}$, there exists an isometry $\Theta_{x} : H \mapsto \ell^{2} (U_{x})$ defined by :

$$ \Theta_{x} \left( \sum_{a} f(a)e_{a} \right) =\left((n+1) \sum_{a \in I_{z}^{n,k,x}}f(a)\right)_{(n,k,z)}. $$

In the same way, we introduce an isometry $\Theta_{x'}$ between $H_{x'}$ and $\ell^{2} (U_{x'}) $.\\

In the rest of the section, we will set $d=d(x,x')$. We will order the geodesic between $x$ and $x'$, in the following sense : $[x,x']=\{x_{0},x_{1},...,x_{d}\}$ with $x_{0}=x$, $x_{d}=x'$ and $d(x_{i},x_{i+1})=1$ for $0\le i \le d-1$.\\

The goal will be to compare the norms of $H_{x}$ and $H_{x'}$. In contrast with Lafforgue's paper \cite{Lafforgue}, we will not write every class $I_{z}^{n,k,x}$ as a disjoint union of $I_{z'}^{n',k',x'}$ but rather as a disjoint union of differences between classes $I_{z'}^{n',k',x'}$.\\

Let $I_{z}^{n,k,x}$ be an equivalence class. 

\begin{proposition}\label{ arbre z sur la géodésique entre x et x'}

Let assume that $z \notin [x,x']$. Then if we set $k'=d(x',z)$, $n'=n-k+k'$ we have that :
$$I_{z}^{n,k,x}=I_{z}^{n',k',x'}.$$

\end{proposition}

\begin{proof}

In this case the proof is the same as Lafforgue's paper \cite{Lafforgue}, let us recall it for the reader's convenience. Let $a \in I_{z}^{n,k,x}$.
The concatenation of the geodesics between $a$ and $z$ and between $z$ and $x'$ is an injective path from $a$ to $x'$ going through $z$. As $X$ is a tree, we can deduce that $z \in [x',a]$. Therefore, $$d(x',a)=d(x',z)+d(z,a)=k'+n-k=n' .$$

Hence $a \in I_{z}^{n',k',x'}$ and  $I_{z}^{n,k,x} \subset I_{z}^{n',k',x'} $. For the reverse inclusion, we apply the same argument.

\end{proof}

Let $a \in I_{z}^{n,k,x}$.
In the rest of the section $t$ will denote the center of $[x,x',a]$ the unique geodesic triangle between $x$, $x'$ and $a$.

\newpage

\begin{figure}[!ht]
    \centering
    \includegraphics[scale=0.6]{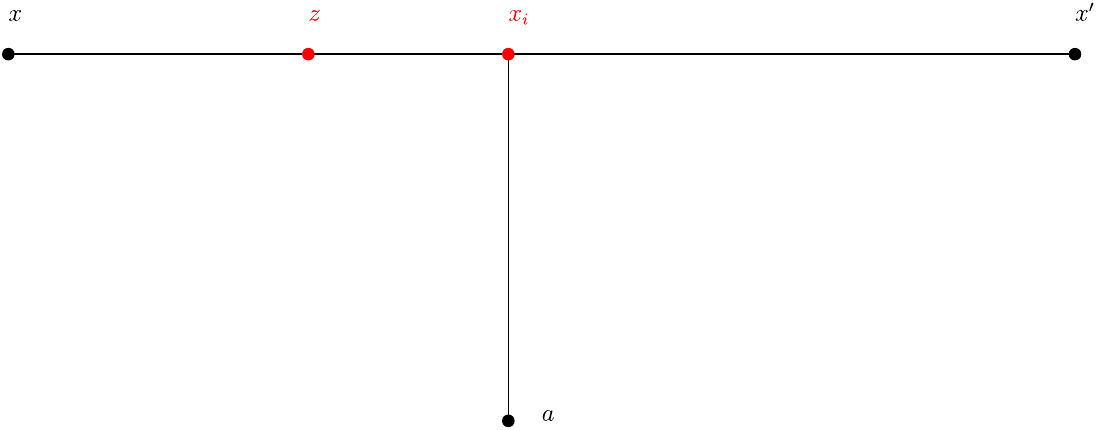}
    \caption{Case where $z$ is in $[x,x']$}
    \label{ fig: Case where $z$ is in $[x,x']$ }
\end{figure}

\begin{proposition}\label{arbre z dans le geod entre x et x'}

Let assume that $z \in [x,t]$, in other words $z \in [x,x']$. There exists $i_{0}\in \mathbb{N}$, $0\le i_{0}\le d$ such that $z=x_{i_{0}}$. Let $i\in \mathbb{N}$, $0\le i\le d$ such that $x_{i}$ is the projection of $a$ on $[x,x']$. Then we have :
\begin{itemize}
    \item$i_{0} \le i \le d$,
    \item $a \in I_{x_{i}}^{n-i+d-i,d-i,x'} - I_{x_{i-1}}^{n-i+d-i,d-i+1,x'}$,
    
    \item $I_{x_{i}}^{n-i+d-i,d-i,x'} - I_{x_{i-1}}^{n-i+d-i,d-i+1,x'} \subset I_{z}^{n,k,x}$, with the convention $I_{x_{-1}}^{n+d,d+1,x'}=\emptyset $.

\end{itemize}

\end{proposition}

\begin{proof}
We will prove the first point by contradiction. Let assume that $0 \le i < i_{0}$. We know that $x_{i} \in [x,a]$. Furthermore, as $x_{i}$ is projection of $a$ on $[x,x']$, $d(x_{i},a) < d(x_{i_0},a)$. Then, as $x_{i_{0}}\in [x,a]$,

$$d(x,a)=d(x,x_{i})+d(x_{i},a)<d(x,x_{i_{0}})+d(x_{i_{0}},a)\le d(x,a) .$$

This is a contradiction, then we know that $i_{0} \le i \le d$.\\

Now we will prove the second point. The concatenation of the geodesic between $a$ and $x_{i}$ and between $x_{i}$ and $x'$ is an injective path between $a$ and $x'$. As $X$ is a tree, we can deduce that $x_{i} \in [x',a]$. Furthermore, as $d(x',x_{i})=d-i$, we have that $$d(x',a)=d(x',x_{i})+d(x_{i},a)=d-i+n-d(x,x_{i})=d-i+n-i. $$

Then $a \in I_{x_{i}}^{n-i+d-i,d-i,x'}$.\\

We will show that $a \notin I_{x_{i-1}}^{n-i+d-i,d-i+1,x'}$. By contradiction, if $a \in I_{x_{i-1}}^{n-i+d-i,d-i+1,x'}$ then $x_{i-1} \in [a,x']$. Furthermore, as $x_{i}$ is the projection of $a$ on $[x,x']$, we know that  $d(a,x_{i})<d(a,x_{i-1})$, then $$d(x',a)=d(x',x_{i})+d(x_{i},a)<d(x',x_{i-1})+d(x_{i-1},a)\le d(x',a) .$$

This is a contradiction then $a \notin I_{x_{i-1}}^{n-i+d-i,d-i+1,x'}$ and we showed that $a \in I_{x_{i}}^{n-i+d-i,d-i,x'} - I_{x_{i-1}}^{n-i+d-i,d-i+1,x'}$.\\

We will prove the third point. Let $ i \in \mathbb{N}$ such that $i_{0} \le i \le d$, we have to show that $I_{x_{i}}^{n-i+d-i,d-i,x'} - I_{x_{i-1}}^{n-i+d-i,d-i+1,x'} \subset I_{z}^{n,k,x}$. Let $a' \in I_{x_{i}}^{n-i+d-i,d-i,x'} - I_{x_{i-1}}^{n-i+d-i,d-i+1,x'}$, we will show that the projection of $a'$ on $[x,x']$ is $x_{i}$. We denote by $x_{l}$ the projection of $a'$ on $[x,x']$ for $l \in \mathbb{N}$ such that $0\le l  \le d$. 
By a previous argument, as $a'\in I_{x_{i}}^{n-i+d-i,d-i,x'}$, we know that $0 \le l \le i$. Let assume by contradiction that $0 \le l < i$, so we have that $ x_{i-1} \in [x_{l},x']$. As $x_{l} \in [x',a']$, we deduce that $x_{i-1} \in [x',a']$. Moreover, we know that $d(x',a')=n-i+d-i$, which implies that $a' \in I_{x_{i-1}}^{n-i+d-i,d-i+1,x'}$, which is a contradiction with the assumption. Then $x_{l}=x_{i}$.

As $x_{i}$ is the projection of $a'$ on $[x,x']$, we know now that $x_{i} \in [x,a']$. We can also remark that $x_{i_0} \in [x,x_{i}]$ as $i_{0}\le i$, we deduce then that $ x_{i_0} \in [x,a'] $.\\

We know also that $d(x',a')=d(x',x_{i})+d(x_{i},a')=d-i+n-i$, where $d(x_{i},a')=n-i$. Thus we remark that :
 $$ d(x,a')=d(x,x_{i})+d(x_{i},a')=i+n-i=n .$$

 Then $a' \in I_{z}^{n,k,x}$ and $I_{x_{i}}^{n-i+d-i,d-i,x'} - I_{x_{i-1}}^{n-i+d-i,d-i+1,x'} \subset I_{z}^{n,k,x}$, we proved the third point.

\end{proof}

\begin{corollary}\label{corollaire decomposition si z dans la geodesique}
If $z \in [x,x']$, i.e. there exists $i_{0} \in \mathbb{N}$, $0 \leq i_{0} \leq d$ such that $z=x_{i_{0}}$, then :
$$I_{z}^{n,k,x}= \bigcupdot \limits_{i=i_{0}}^{d} (I_{x_{i}}^{n-i+d-i,d-i,x'} - I_{x_{i-1}}^{n-i+d-i,d-i+1,x'}) $$
    with the convention that $I_{x_{-1}}^{n+d,d+1,x'}=\emptyset $, where $\bigcupdot$ denote a disjoint union.

\end{corollary}

\begin{proof}
The inclusion $I_{z}^{n,k,x} \subset \bigcup \limits_{i=i_{0}}^{d} (I_{x_{i}}^{n-i+d-i,d-i,x'} - I_{x_{i-1}}^{n-i+d-i,d-i+1,x'}) $ is clear with the second point of Proposition \ref{arbre z dans le geod entre x et x'}.\\

The reverse inclusion $\bigcup\limits_{i=i_{0}}^{d} (I_{x_{i}}^{n-i+d-i,d-i,x'} - I_{x_{i-1}}^{n-i+d-i,d-i+1,x'}) \subset I_{z}^{n,k,x} $ is also clear with the third point of Proposition \ref{arbre z dans le geod entre x et x'}, thus $$ I_{z}^{n,k,x} = \bigcup \limits_{i=i_{0}}^{d} (I_{x_{i}}^{n-i+d-i,d-i,x'} - I_{x_{i-1}}^{n-i+d-i,d-i+1,x'}) .$$

It remains to show that the inclusion is disjoint. In the proof of Proposition \ref{arbre z dans le geod entre x et x'}, we proved that the projection of the points in $I_{x_{i}}^{n-i+d-i,d-i,x'} - I_{x_{i-1}}^{n-i+d-i,d-i+1,x'}$ on $[x,x']$ is $x_{i}$, therefore the union is disjoint.

\end{proof}

Now we are ready to compare the norms of $H_{x}$ and $H_{x'}$. We will show that for all $ (n,k,z) \in U_{x} $, the class $I_{z}^{n,k,x}$ admits a decomposition with a uniformly bounded number of classes $I_{z'}^{n',k',x'}$ for $(n',k',z') \in U_{x'}$ and conversely every class $I_{z'}^{n',k',x'}$ appears in a uniformly bounded number of decompositions of classes $I_{z}^{n,k,x}$.
Both points are corollaries of the previous properties.

\begin{proposition}\label{arbredenombrement1}

For all $(n,k,z) \in U_{x} $, the number of classes $I_{z'}^{n',k',x'}$ for $(n',k',z') \in U_{x'} $ which appear in the decomposition of the class $I_{z}^{n,k,x}$ is bounded above by $2d(x,x')+2$.

\end{proposition}

\begin{proof}

Let $(n,k,z) \in U_{x} $.

If $z \notin [x,x']$ according to Proposition \ref{ arbre z sur la géodésique entre x et x'}, we need only one class $I_{z'}^{n',k',x'}$ to decompose $I_{z}^{n,k,x}$. \\

If $z\in [x,x']$, we set $z=x_{i_{0}}$. According to Proposition \ref{arbre z dans le geod entre x et x'}, there are $2(d-i_{0}+1)$ classes $I_{z'}^{n',k',x'}$ which appear in the decomposition. By noticing that $d-i_{0}<d(x,x')$, we proved the proposition.

\end{proof}

For the reverse enumeration, we get the following proposition :

\begin{proposition}\label{arbredenombrement2}

For each $(n',k',z') \in U_{x'} $, the class $I_{z'}^{n',k',x'}$ appears in the decompositions of at most $2d(x,x')+3$ classes $I_{z}^{n,k,x}$ for $(n,k,z) \in U_{x} $.

\end{proposition}

\begin{proof}

If $z' \notin [x,x']$ according to Propositions \ref{ arbre z sur la géodésique entre x et x'} and \ref{arbre z dans le geod entre x et x'} , the class $I_{z'}^{n',k',x'}$ appears in the decomposition of one class, namely $I_{z'}^{n,k,x}$ for $k=d(x,z')$ and $n=n'-k'+k$. \\

If $z' \in [x,x']$, we set $z'=x_{i_{0}}$. According to proposition \ref{ arbre z sur la géodésique entre x et x'} and \ref{arbre z dans le geod entre x et x'}, the class $I_{z'}^{n',k',x'}$ appears positively in $i_{0}+1$ classes $I_{z}^{n,k,x}$. In fact, the possible values for $z$ are the $x_{i}$ for $0 \le i \le i_{0}$ and for each possible $z$ there is only one choice of $k$ and $n$. As $i_{0}+1=d(x,z)+1$, it appears in at most $d(x,x')+1$ decompositions of classes $I_{z}^{n,k,x}$.

In the same way, this class $I_{z'}^{n',k',x'}$ appears negatively in at most $i_{0}+2$ classes $I_{z}^{n,k,x}$. The possible values for $z$ are the $x_{i}$ for $0 \le i \le i_{0}+1$ and for each $z$ there is a unique choice of $k$ and $n$, so at most $d(x,x')+2$ classes $I_{z}^{n,k,x}$.

To conclude, the class $I_{z'}^{n',k',x'}$ appears in at most $2d(x,x')+3$ decompositions of classes  $I_{z}^{n,k,x}$.
    
\end{proof}

Let $A$ be the matrix from $\ell^{2} (U_{x'}) $ to $\ell^{2} (U_{x}) $ with coefficient $\frac{n+1}{n'+1}$ if $I_{z'}^{n',k',x'}$ appears positively in the decomposition of $I_{z}^{n,k,x}$, $-\frac{n+1}{n'+1}$ if $I_{z'}^{n',k',x'}$ appears negatively in the decomposition of $I_{z}^{n,k,x}$ and $0$ otherwise. According to Propositions \ref{ arbre z sur la géodésique entre x et x'} and \ref{arbre z dans le geod entre x et x'}, we remark that $A \circ \Theta_{x'}=\Theta_{x}$. The next proposition gives a bound on the norm of $A$.\\

\begin{proposition}\label{boundonthenorm of A}

Let $A$ be the matrix defined above. Then, we have that $$ \vert\vert A\lvert\lvert \le \sqrt{2d(x,x')+2} \sqrt{2d(x,x')+3} (d(x,x')+1) .$$

\end{proposition}

\begin{proof}

Let $I_{z'}^{n',k',x'}$ be a class in the decomposition of $I_{z}^{n,k,x}$ and $a \in I_{z}^{n,k,x} \cap I_{z'}^{n',k',x'}$, according to the reverse triangle inequality, we remark that : $$ \vert n-n' \lvert = \vert d(x,a)-d(x',a) \lvert \le d(x,x'). $$

From this inequality, we get a bound on the coefficients of $A$, indeed $\frac{n+1}{n'+1} \le d(x,x')+1$. According to Propositions \ref{arbredenombrement1} and \ref{arbredenombrement2}
, in every line of $A$ there are at most $2d(x,x')+2$ non-zero coefficients and in every column at most $2d(x,x')+3$ non-zero coefficients. From this, we deduce the following bound of the norm of $A$ :  

$$ \vert\vert A\lvert\lvert \le \sqrt{2d(x,x')+2} \sqrt{2d(x,x')+3} (d(x,x')+1) .$$

\end{proof}

According to Proposition \ref{boundonthenorm of A}, we have a polynomial control in $d(x,x')$ of the norm of $A$. Thus for all $f \in \mathbb{C}^{(X)}$, we have that $ \vert \vert f \lvert \lvert_{H_{x}} \le \vert \vert A \lvert \lvert ~ \vert \vert f \lvert \lvert_{H_{x'}}$ and then $ \vert \vert \pi(g) f \lvert \lvert_{H} \le \vert \vert A \lvert \lvert ~ \vert \vert \pi(g)f \lvert \lvert_{H_{x'}}$.

With the fact that, for all $f\in \mathbb{C}^{(X)}$, for all $g\in G$, $\vert \vert \pi(g)f \lvert \lvert_{H_{x'}}= \vert \vert f \lvert \lvert_{H} $, as $x'=g(x)$, we deduce that $\pi(g)$ is continuous from $H_{x}$ dans $H_{x}$ and its norm is bounded by $\sqrt{2d(x,x')+2} \sqrt{2d(x,x')+3} (d(x,x')+1)=\sqrt{l(g)+2} \sqrt{l(g)+3} (l(g)+1)$.\\

 To summarize, we have built a linear continuous representation from $G$ on an Hilbert space which is sub-exponential. To conclude, we will show that the map $\sum_{a} f(a)e_{a} \mapsto \sum_{a} f(a)$ is a continous linear form.

 \begin{proposition}\label{formelinéairecontinue}

 The map : $$\phi : \sum_{a} f(a)e_{a} \mapsto \sum_{a} f(a)$$ from $H$ to $\mathbb{C}$ is a continuous linear form.

\end{proposition}

\newpage

\begin{proof}

Let $f \in H$.

$$\begin{aligned} 
           \biggr\vert \sum_{a} f(a) \biggr\lvert^{2}&= \biggr\vert \sum_{n \in \mathbb{N}}\sum_{a \in S_{x}^{n}} f(a) \biggr\lvert^{2} \\
            & \leq \bigg(\sum_{n \in \mathbb{N}} \frac{1}{(n+1)^2}\bigg) \bigg(\sum_{n \in \mathbb{N}} (n+1)^{2}\biggr\vert\sum_{a \in S_{x}^{n}} f(a) \biggr\lvert^{2} \bigg) ~(\text{according to  the Cauchy-Schwarz inequality})\\      
            & \leq \sum_{n \in \mathbb{N}} \frac{1}{(n+1)^2} \bigg(\sum_{n \in \mathbb{N}} (n+1)^{2} \sum_{k=0}^{n} \sum_{z \in S_{x}^{n} } { \biggr\vert \sum_{a \in I_{z}^{n,k,x}} f(a) \biggr\lvert }^{2}\bigg)~(\text{since } I_{x}^{n,0,x}=S_{x}^{n})\\
            & \leq \sum_{n \in \mathbb{N}} \frac{1}{(n+1)^{2}} \left\| f \right\|_{H_{x}}^{2}
\end{aligned}$$

Then, the linear form $\phi$ is continuous with $ ||| \phi |||^{2} \leq \sum_{n \in \mathbb{N}} \frac{1}{(n+1)^{2}} < \infty $.

\end{proof}

The representation $\pi$ on $H$ is sub-exponential. It preserves also the linear form $\phi$, which is continuous according to Proposition \ref{formelinéairecontinue}. As $G$ has the strong Property $(T)$, according to Proposition \ref{Proposition : fixed point for strong Property (T)}, this representation should admit a non-zero invariant vector.

The following proposition allows us to conclude the proof of Theorem \ref{theoremeimportant arbre} :

\begin{proposition}\label{Proposition : Orbites bornées}
If the representation $H$ admits a non-zero $G$-invariant vector, then every orbit for the action of $G$ on $X$ is bounded.

\end{proposition}

\begin{proof}
The first point is the fact that if one orbit is bounded, then every orbit is bounded. Let $x \in X$; assume that the orbit of $x$ by $G$ is bounded. Let $y \in X$, for all $g \in G$, we have :

$$d(y,g.y) \le d(y,x)+d(x,g.x)+d(g.x,g.y) \le 2d(x,y)+ diam(G.x) . $$

Thus the orbit of $y$ is bounded.

Now, we need to prove that one orbit is bounded.

Let $f$ be a non-zero invariant vector of $H$, for all $g \in G$, for all $x \in X$, $$\pi(g)(f)(x)=f(g^{-1}x)=f(x). $$

Then, $f$ is constant on every orbit of the action of $G$ on $H$.

Since $f$ is non-zero, there exists $x_{0} \in X$ such that $f(x_{0})\neq0$. By definition, $H$ is the set of  functions $f$ such that the sum :
$$ \left\| \sum_{a} f(a)e_{a} \right\|_{H_{x}}^{2}= \sum_{n \in \mathbb{N}} (n+1)^{2} \sum_{k=0}^{n} \sum_{z \in S_{x}^{k} } { \biggr\vert \sum_{a \in I_{z}^{n,k,x}} f(a) \biggr\lvert }^{2}$$ is finite.

This norm is finer than the $\ell^{2}$ norm, indeed :
$$\begin{aligned} 
           \left\| \sum_{a} f(a)e_{a} \right\|_{H_{x}}^{2} & = \sum_{n \in \mathbb{N}} (n+1)^{2} \sum_{k=0}^{n} \sum_{z \in S_{x}^{n} } { \biggr\vert \sum_{a \in I_{z}^{n,n,x}} f(a) \biggr\lvert }^{2}\\      
            & \geq \sum_{n \in \mathbb{N}} (n+1)^{2} \sum_{z \in S_{x}^{n} } { \biggr\vert \sum_{a \in I_{z}^{n,k,x}} f(a) \biggr\lvert }^{2}\\
            & \geq \sum_{n \in \mathbb{N}} (n+1)^{2} \sum_{z \in S_{x}^{n} } { \vert f(z) \lvert }^{2}.
\end{aligned}$$

Therefore, $f$ cannot be constant on an infinite set. The orbit $G.x_{0}$ is then finite, thus bounded, and according to the first point every orbit is bounded.

\end{proof}

\section{Relatively hyperbolic groups do not have the strong Property $(T)$}\label{section : Preuve groupes relativement hyperbolic }

In this section, we will adapt Lafforgue's proof in \cite{Lafforgue} to generalize Lafforgue theorem for relatively hyperbolic groups and prove Theorem \ref{theorem:main result}.

Let $\delta \ge 0$, in this section $X$ will denote a uniformly fine $\delta$-hyperbolic graph and $G$ a group with the strong Property $(T)$ equipped with an isometric action on $X$.\\
For all $n \in \mathbb{N}$ and for all  $x \in X$, we denote $S_{x}^{n}$ the sphere centered at $x$ and of radius $n$. For all $k,n \in \mathbb{N}$, we introduce the following partition of $S_{x}^{n}$ for the equivalence relation :
$a\mathcal{R}b$ if for all $z\in B(x,k)$, we have $d(a,z)=d(b,z)$. We denote $J^{n,k,x}$ the set of equivalence classes of this equivalence relation and $S_{x}^{n}=\bigcupdot_{i\in J^{n,k,x}} I_{i}^{n,k,x}$, where the $I_{i}^{n,k,x}$ denote the equivalence classes. Note that, for a fixed $x\in X$, the equivalences classes $I_{i}^{n,k,x}$ are either disjoint or contained in one another. \\

We denote $\mathbb{C}^{(X)}$ the vector space of functions from $X$ to $\mathbb{C}$ with finite support. Then, for each $x\in X$, we define the Hilbert space $H_{x}$, which is the completion of $\mathbb{C}^{(X)}$ for the following norm :

$$ \left\| \sum_{a} f(a)e_{a} \right\|_{H_{x}}^{2}= \sum_{n \in \mathbb{N}} (n+1)^{2} \sum_{k=0}^{n} \sum_{i \in J^{n,k,x} } { \biggr\vert \sum_{a \in I_{i}^{n,k,x}} f(a) \biggr\lvert }^{2}.$$\\

For each $x \in X$, we define a length on $G$ $l_{x}$, by $l_{x}(g)=d(x,g.x)$ for all $g \in G$.\\

We will show the following results :

\begin{itemize}
    \item the action of $G$ by left translations on $\mathbb{C}^{(X)}$, given by the formula  $\pi(g)(f)(x)=f(g^{-1}x)$ for $g \in G$, $f \in \mathbb{C}^{(X)}$, $x \in X$, admits an extension to a continuous representation of $G$ on $H=H_{x}$ for some $x \in X$,

    \item there exists a polynomial $P$ such that $||\pi(g)||_{\mathcal{L}(H)} \leq P(l_{x}(g))$ for all $g \in G$, for some $x \in X$,

    \item the dual representation $H^{\ast}$ admits a non-zero invariant vector,

    \item finally $H$ admits a non-zero vector if and only if the orbits of the action of $G$ on $X$ are bounded.
\end{itemize}

For all $x \in X$, we set $U_{x} = \{(n,k,i), 0 \leq  k\leq n, i\in J^{n,k,x} \}$.
Let $g \in G$, we set $x'=g(x)$ such that $d(x,x')=l_{x}(g)$. By definition of $H=H_{x}$, there exists an isometry $\Theta_{x} : H \mapsto \ell^{2} (U_{x})$  defined by :

$$ \Theta_{x} \left( \sum_{a} f(a)e_{a} \right) =\left((n+1) \sum_{a \in I_{i}^{n,k,x}}f(a)\right)_{(n,k,i)}. $$

In the same way, we introduce $\Theta_{x'}$ an isometry between $H_{x'}$ and $\ell^{2} (U_{x'}) $.\\

In the rest of the section, we fix $x,x' \in X$, such that $x'=g(x)$, the goal will be to compare the norms $H_{x}$ and $H_{x'}$. In contrast with Lafforgue's paper \cite{Lafforgue}, we will not write every class $I_{i}^{n,k,x}$ as a disjoint union of $I_{i'}^{n',k',x'}$ but rather as a disjoint union of differences between classes $I_{i'}^{n',k',x'}$. In the rest, of the section, we will denote by $l$ the length $l_{x}$. \\

We remark that the two following lemmas of Lafforgue's article \cite{Lafforgue}, which allow to describe the equivalence classes $I_{i}^{n,k,x}$, are still true for uniformly fine graphs. In fact their proofs only use the hyperbolicity of the graph without any finiteness assumption like uniform local finiteness. Therefore we will use the following lemmas in the rest of the section :

\begin{lem} \label{classedequnpoint} \cite[Lemme 1.8]{Lafforgue}

Let $a\in X$,  $n=d(x,a)$, $k\in \{0,...,n\}$ and $z$ a vertex in $B(x,k)$
at minimal distance of $a$ (i.e. $n-k$).
Then the knowledge of the distances from $a$ to the points of  $B(x,k)\cap B(z,3 \delta) $ implies the knowledge of the distances from $a$ to the points of $B(x,k)$.

\end{lem}

\begin{lem}\label{classedeqdeuxpoint} \cite[Lemme 1.9]{Lafforgue}

Let $a\in X$,  $n=d(x,a)$, $k\in \{0,...,n\}$ and $z$ a vertex in $B(x,k)$
at minimal distance of $a$ (i.e. $n-k$).
If $a'$ is another vertex of $X$, such that $a$ and $a'$ are at the same distance to the points of $B(x,k)\cap B(z,3\delta)$, then $a'$ belongs the same class $I_{i}^{n,k,x}$ as $a$.

\end{lem}

Nevertheless, for $X$ a uniformly locally finite graph, the set $B(x,k)\cap B(z,3\delta)$ is uniformly finite. Then in a uniformly finite hyperbolic graph a corollary of Lemma \ref{classedequnpoint} is the fact that if we take $a\in S_{x}^n$ and $z$ a point at minimal distance of $a$ in $B(x,k)$, $a$ could be included in a finite number of classes $I_{i}^{n,k,x}$, which is a function of $ \delta$. Therefore to enumerate the classes $I_{i'}^{n',k',x'}$ in the decomposition of $I_{i}^{n,k,x}$, we can enumerate the number of $z'$ associated to the classes $I_{i'}^{n',k',x'}$. When $X$ is no longer uniformly finite, we will prove a similar lemma adjusted to our setting. \\

\begin{proposition}\label{conereldeq}
Let $a\in X$,  $n=d(x,a)$, $k\in \{0,...,n\}$ and $z$ a vertex in $B(x,k)$
at minimal distance of $a$ (i.e. $n-k$), $\gamma$ a geodesic between $x$ and $a$ and $e$ the edge of $\gamma$ between $x$ and $z$ containing $z$. Then the knowledge of the distances from $a$ to the points of  $Cone_{50\delta}(e) $ implies the knowledge of the distances from $a$ to the points of $B(x,k)$.
\end{proposition}

\begin{figure}[!ht]
    \centering
    \includegraphics[scale=0.6]{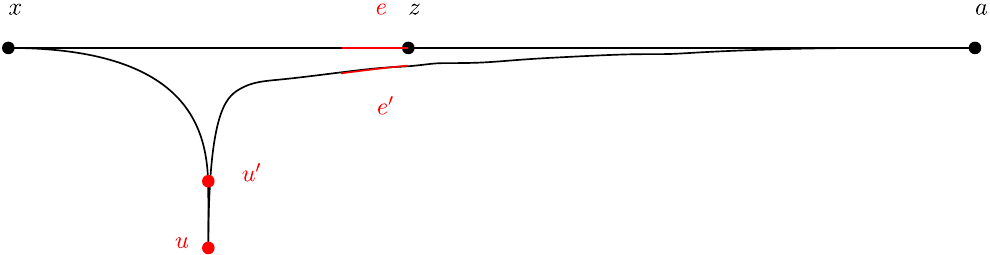}
    \caption{ Proof of Proposition \ref{conereldeq} in the case where $e' \subset [u,a]$}
    \label{ Proof of Proposition conereldeq in the case where e' subset [u,a] }
\end{figure}

\begin{proof}

Let $u \in B(x,k)$ and consider a geodesic triangle $[x,u,a]$ such that $[x,a]=\gamma$. According to Proposition \ref{conetriangle}, $e$ is contained in a cone of parameter $50\delta$ around an edge $e'$ that is either on $[u,a]$ at same distance of $a$ than $e$, either on $[x,u]$ at same distance of $x$ than $e$.

Let assume that $e'\in [x,u]$ then $d(x,e')=d(x,e)=k-1$. Since $[x,u]\subset B(x,k)$ then $u\in e'$. From this we deduce that $u\in Cone_{50\delta}(e)$ and then the distance from $a$ to $u$ is known.

Let assume that $e' \in [u,a]$. If $u \in Cone_{50\delta}(e)$ then the result is clear. Otherwise, we take $u'\in [u,a] $ such that $d(a,u)=d(a,u')+1$ to get a point on $[u,a]$ closer to $a$ than $u$, such that the knowledge of the distance from $a$ to $u'$ implies the knowledge of the distance from $a$ to $u$. In the case, where $u' \in Cone_{50\delta}(e)$, the result is proved. Otherwise, we can iterate the process until we get a vertex of $e'$, in other words a point in $Cone_{50\delta}(e)$. This ends the proof of the proposition.

\end{proof}

As cones in $X$ are uniformly finite,
Proposition \ref{conereldeq} leads to the following corollary: 

\begin{corollary} \label{denombrement relation dequivalence cones}
Let $a\in X$,  $n=d(x,a)$, $k\in \{0,...,n\}$ and $z$ a vertex in $B(x,k)$
at minimal distance of $a$ (i.e. $n-k$). Then the number of classes $I_{i}^{n,k,x}$ in which $a$ could be included for fixed $z$, is bounded above by a function depending only on $\delta$ and the uniform finesse function $\varphi$. We will denote $L(\delta,\varphi)$ this function and we will sometimes use the abuse of notation $L$.

\end{corollary}

Now we have all the tools to prove the decomposition of $I_{i}^{n,k,x}$ as a disjoint union of classes $I_{i'}^{n',k',x'}$ and differences between classes $I_{i'}^{n',k',x'}$. The key property will be the following one. We will say that a class $I_{i'}^{n',k',x'}$ appears negatively if we pull this class out of the decomposition and appears positively otherwise. Its proof will be given in subsection \label{subsec : conclusion}  using the result of subsections \ref{subsec : premier cas}, \ref{subsec : deuxième cas}, \ref{subsec : troisième cas} and \ref{subsec : quatrième cas}.

\begin{proposition} \label{propriété clé décomposition}
For each $(n,k,i) \in U_{x}$, the class $I_{i}^{n,k,x}$ can be written as a finite disjoint union of classes $I_{i'}^{n',k',x'}$ or differences between classes $I_{i'}^{n',k',x'}$ where  $(n',k',i') \in U_{x'} $.\\  

Furthermore, there exists a constant $K$ depending only on $\delta$ and $f$  such that the number of classes $I_{i'}^{n',k',x'}$ , which appear positively in the decomposition of  $I_{i}^{n,k,x}$ is bounded above by $K(d(x,x')+1)$. In the same way, the number of classes $I_{i'}^{n',k',x'}$ which appear negatively is bounded above by $K(d(x,x')+1)$.\\

Conversely, each class $I_{i'}^{n',k',x'}$ appears in at most $K(d(x,x')+1)$ decompositions of classes $I_{i}^{n,k,x}$.

\end{proposition}

Let $a \in I_{i}^{n,k,x}$, by definition of the classes $I_{i}^{n,k,x}$, the set of points in $B(x,k)$ at minimal distance of $a$ does not depend on $a$. So we can talk about the points of $B(x,k)$ at minimal distance of $I_{i}^{n,k,x}$. We will work in a geodesic triangle $[x,x',a]$ between the points $x$, $x'$ and $a$ given by Theorem \ref{formenormaledestriangles}. We introduce the points $\tilde{x}$, $\tilde{x'}$ et $\tilde{a}$ defined as in Definition \ref{pointloingrandangle}. We denote $z$ the point on $B(x,k)$ at minimal distance of $a$ on the triangle, i.e. the point on the side $[x,a]$ of the triangle such that $d(x,z)=k$ and $d(z,a)=n-k$.

We will prove the decompostion of the class $I_{i}^{n,k,x}$ by exhaustion according to the position of $z$ in the triangle.

\subsection{ The case $z \in [\tilde{a},a]$ and $z \neq \tilde{a} $ } \label{subsec : premier cas}

\begin{proposition}  \label{ cas z  dans [tilde a, a]}

Let us assume that $z \in [\tilde{a},a] $ and $z\neq \tilde{a}$ except if  $a=\tilde{a}$. Let $k'=d(x',z)$, $n'=d(x',a)$ and $I_{i'}^{n',k',x'}$ the equivalence class of $a$ for $x'$, $n'$, $k'$ then :
 $$I_{i}^{n,k,x}=I_{i'}^{n',k',x'}. $$

\end{proposition}

\begin{figure}[!ht]
    \centering
    \includegraphics[scale=0.62]{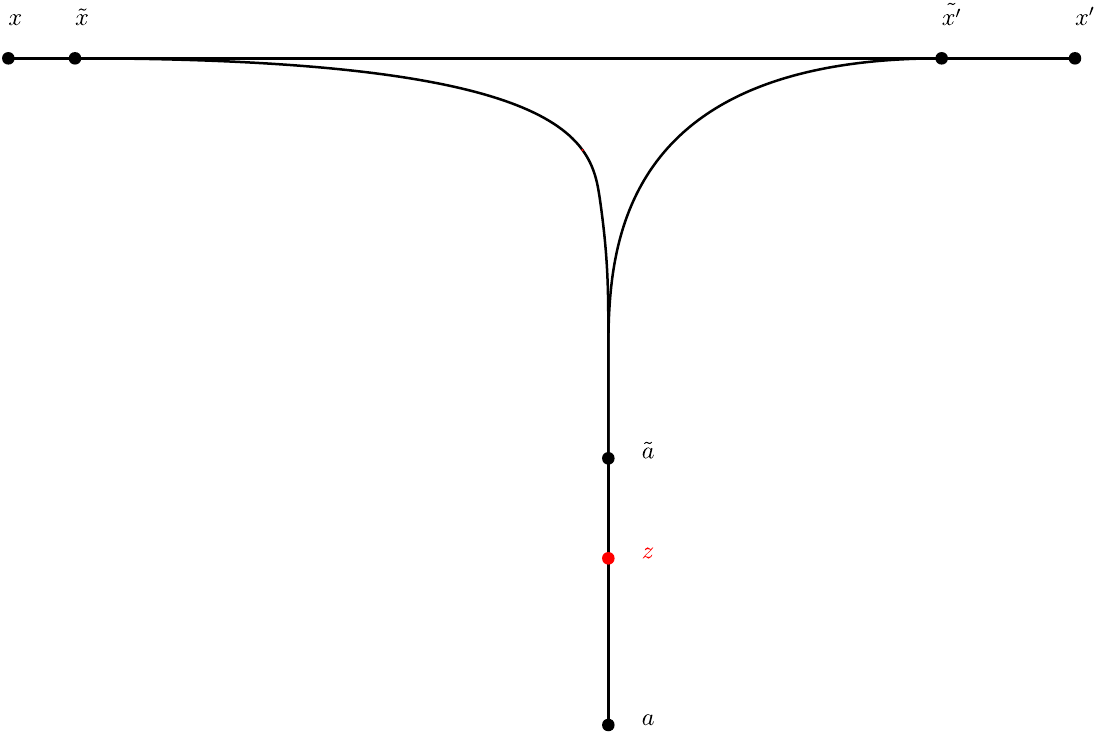}
    \caption{Case where $z$ is between $a$ and $\tilde{a}$ }
    \label{ fig: case where z is in a et tildea }
\end{figure}

\begin{proof}

We will show that $I_{i'}^{n',k',x'} \subset I_{i}^{n,k,x}$. Let $a' \in I_{i'}^{n',k',x'} $, we will then show that for all $y \in B(x,k)$, $d(a',y)=d(a,y)$.  \\

Let $y\in B(x,k)$. We assume first that $\tilde{a} \in I(x,y)$. In this case we will show that $y \in B(x,k')$. In fact : \\

$$\begin{aligned} 
            d(y,\tilde{a})& = d(x,y)-d(\tilde{a},x) \\      
            & \leq k- d(x,z)+d(\tilde{a},z) \\
            &\leq k-k+d(\tilde{a},z)\\
            &\leq d(\tilde{a},z).
\end{aligned}$$

Therefore :\\

$$\begin{aligned} 
            d(y,x')& \leq d(y,\tilde{a})+d(\tilde{a},x') \\      
            & \leq d(z,\tilde{a})+k'-d(\tilde{a},z) \\
            &\leq k'. 
\end{aligned}$$

Then $y    \in B(x,k')$ so $d(a,y)=d(a',y)$, since $a,a' \in I_{i'}^{n',k',x'} $.\\

Now we assume that $\tilde{a} \notin I(x,y)$. We will show that $\tilde{a} \in I(a',y)$. By contraposition of Proposition \ref{angletoutegéod}, we have $\measuredangle_{\tilde{a}}(x,y) \leq 12\delta$.  

 Then thanks to Proposition \ref{Triangle inequality}, we have that :

$$\begin{aligned} 
            \measuredangle_{\tilde{a}}(a,x)& \leq \measuredangle_{\tilde{a}}(a,a')+\measuredangle_{\tilde{a}}(a',x) \\      
            & \leq \measuredangle_{\tilde{a}}(a,a')+ \measuredangle_{\tilde{a}}(a',y)+ \measuredangle_{\tilde{a}}(y,x) .
\end{aligned}$$\\

Yet we know with Definition \ref{pointloingrandangle} of $\tilde{a}$, that $\measuredangle_{\tilde{a}}(a,x) \geq 50\delta$. Moreover $d(a,a') \le d(a,z)+d(z,a') < d(a,\tilde{a})+d(\tilde{a},a') $ then $\measuredangle_{\tilde{a}}(a,a')< 12 \delta$ since $\tilde{a} \notin I(a,a')$. Then :

$$ \measuredangle_{\tilde{a}}(a',y) \geq \measuredangle_{\tilde{a}}(a,x)-\measuredangle_{\tilde{a}}(y,x)-\measuredangle_{\tilde{a}}(a,a') \geq 36 \delta > 12\delta .$$

We deduce that $\tilde{a} \in I(a',y)$, so $d(a',y)=d(a',\tilde{a})+d(\tilde{a},y)$.\\

In the same way, we remark that $\tilde{a} \in I(a,y)$. In fact,
$$ \measuredangle_{\tilde{a}}(a,y)\geq \measuredangle_{\tilde{a}}(a,x) - \measuredangle_{\tilde{a}}(x,y) > 12\delta .$$

Then, we have $d(a,y)=d(a,\tilde{a})+d(\tilde{a},y)$. We have $\tilde{a} \in B(x',k')$, since $a,a' \in I_{i'}^{n',k',x'}$, we have $d(a,\tilde{a})=d(a',\tilde{a})$, and so $d(a,y)=d(a',y)$.   \\

We proved that $I_{i'}^{n',k',x} \subset I_{i}^{n,k,x}$. We deduce the reverse inclusion by symmetry between the roles of $I_{i'}^{n',k',x}$ and $I_{i}^{n,k,x}$.
 
\end{proof}

In the rest of the section, we denote by $u$ the point of the side $[x,a]$ of the triangle such that $d(x,u)=(a,x')_{x}$ and $d(a,u)=(x,x')_{a}$, $v$ the point of the side $[x',a]$ of the triangle such that $d(x',v)=(a,x)_{x'}$ and $d(a,v)=(x,x')_{a}$. We will denote by $t$ a quasi-center of the triangle $[x,x',a]$ as in Proposition \ref{quasicentreparticulier} such that $d(t,u) \leq \delta$ and $d(t,v) \leq \delta$.\\

Let $\theta>0$, we denote $Cone_{\theta}([x,x'])=\bigcup_{e \in [x,x']} Cone_{\theta}(e)$.\\

Now, we assume that $z \in [x,\tilde{a}]$.
We have to prove the decomposition in three other cases :

\begin{itemize}
    \item the case where $z\in [x,u]$ and $\tilde{a}$ is far enough from $v$ in Subection \ref{subsec : deuxième cas},
    
    \item the case where $z\in [u,\tilde{a}]$ and $\tilde{a}$ is far enough from $z$, in Subsection \ref{subsec : troisième cas},
    
     \item and finally the case where $\tilde{a}$ is too close from $v$ or too close from $z$, in Subsection \ref{subsec : quatrième cas}.

\end{itemize}

\subsection{The case where $z\in [x,u]$ and $\tilde{a}$ is far enough from $v$} \label{subsec : deuxième cas}

\begin{proposition} \label{zentre aetu facile}

Assume that $k \leq (a,x')_{x}$, i.e. $z\in [x,u]$, and that $d(x',\tilde{a}) \geq d(x',v) + 8\delta$. Then if we set $n'=d(x',a)$, $k'=d(x',v)+8\delta$ and $I_{i'}^{n',k',x'}$ the equivalence class of $a$ for $x'$, $n'$ and $k'$ then
 $I_{i'}^{n',k',x'} \subset I_{i}^{n,k,x}$.\\

\end{proposition}

\begin{figure}[!ht]
    \centering
    \includegraphics[scale=0.7]{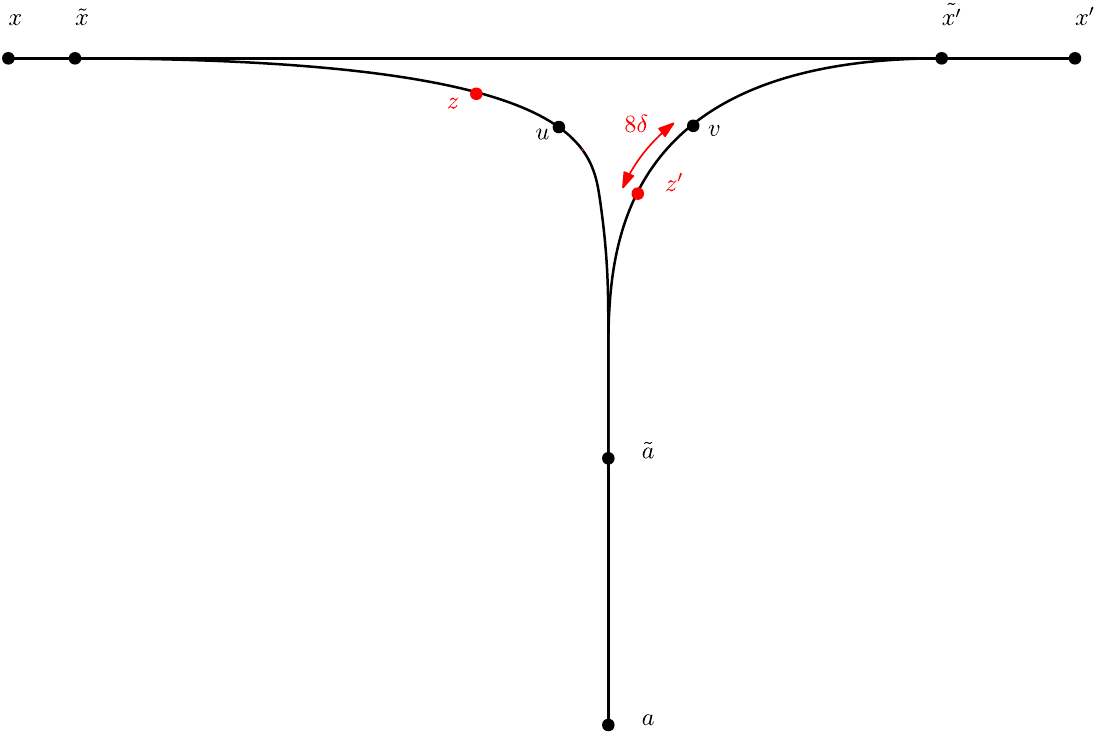}
    \caption{Case where $z$ is between $x$ and $u$ and $\tilde{a}$ is far enough }
    \label{ fig: cas oÃ¹ z est avant u ( premier dessin) }
\end{figure}

\begin{proof}

By assumption, we have $d(a,z)\ge d(a,u)$.\\

If $d(z,u)\leq 3\delta$ then $B(z,3\delta) \subset B(v,8\delta)$ so $B(z,3\delta) \subset B(x',k')$ and according to Lemma \ref{classedeqdeuxpoint}, we deduce that $I_{i'}^{n',k',x'} \subset I_{i}^{n,k,x}$.\\

If $d(z,u)> 3\delta$, we still have that $B(u,3\delta) \subset B(x',k')$. Let $a' \in I_{i'}^{n',k',x'} $, we will show that $a' \in I_{i}^{n,k,x} $. Therefore, the distances of $a$ to $B(u,3\delta)$ and $a'$ to $B(u,3\delta)$ are the same since $a,a' \in I_{i'}^{n',k',x'} $. As $u \in I(a,z)$, $u$ is a point of $B(z,d(z,u))$ at minimal distance of $a$. According to Lemma \ref{classedeqdeuxpoint}, since the distances of $a$ to $B(u,3\delta)$ and $a'$ to $B(u,3\delta)$ are the same, then the distances of $a$ to $B(z,d(z,u))$ and $a'$ to $B(u,d(z,u))$ are the same. Finally, with the remark that $B(z,3\delta) \subset B(z,d(z,u)) $, we can reuse Lemma \ref{classedeqdeuxpoint} to prove that the distances of $a$ and $a'$ to $B(x,k)$ are the same, so $I_{i'}^{n',k',x'} \subset I_{i}^{n,k,x}$.  \\

\end{proof}

Now we will count the number of classes $I_{i'}^{n',k',x'}$ involved in this case.

\begin{proposition} \label{denombrement z dans [x,u] et atilde loin}

Under the assumptions of Proposition \ref{zentre aetu facile}, the number of classes
$I_{i'}^{n',k',x'}$ in the decomposition is bounded above by $L |Cone_{150\delta}([x,x'])|$, where $L$ is the number depending on $\delta$ and $\varphi$ introduced in Corollary \ref{denombrement relation dequivalence cones}.

\end{proposition}

\begin{proof}

To control the number of classes $I_{i'}^{n',k',x'}$ in the decomposition of $I_{i}^{n,k,x}$, we will control the number of $z'$ such that $z' \in [x',a]$ and $d(x',z')=k'$. As such points $z'$ are points of $B(x',k')$ at minimal distance of $a$, we will conclude thanks to Corollary \ref{denombrement relation dequivalence cones}.\\

Let assume that $d(v,\tilde{x'})> 8 \delta$. We denote by $e$ the edge of the side $[v,x']$ such that $d(v,e)=8\delta$. By definition of $z'$, we have $d(z',e)=16\delta$. Furthermore $z' \in [\tilde{a},\Tilde{x'}] $ since $d(x',\tilde{a}) \geq d(x',v) + 8\delta$ and $e$ is an edge of $[v,\tilde{x'}]\subset [\tilde{a},\Tilde{x'}]$. According to Theorem \ref{formenormaledestriangles}, along the side $[x',a]$ of the triangle, from $v$ to $e$, angles are bounded above by $100\delta$. Then we deduce that $z' \in Cone_{100\delta}(e)$. According to Proposition \ref{conetriangleplusfort}, as $d(v,e)> 8 \delta$, there exists an edge $e'$ of the side $[x,x']$ such that $e \in Cone_{50\delta}(e')$, we can deduce that $z' \in Cone_{150\delta}(e') \subset Cone_{150\delta}([x,x']) $. Thus the number of $z'$ is bounded above by $|Cone_{150\delta}([x,x'])|$. According to Corollary \ref{denombrement relation dequivalence cones}, the number of equivalence classes of $a$ for fixed $n',k',x'$ and $z'$ is bounded by $L$ and this leads to the conclusion of the proposition in this case. \\

Let assume that $d(v,\tilde{x'})\le8\delta$. Now we denote by $e$ the edge of the side $[a,x']$ between $v$ and $\tilde{x'}$, which contains $\tilde{x'}$. Since $z'\in [\tilde{a},\tilde{x'}]$ and $e$ is an edge of $[\tilde{a},\tilde{x'}]$, angles from $z'$ to $e$ along the triangle are bounded by $100\delta$, thus $z' \in Cone_{100\delta}(e)$. According to Proposition \ref{conetriangleplusfort} in the triangle $  [ \tilde{x},\tilde{x'} , \tilde{a}] $, $e \in Cone_{50\delta}(e')\subset Cone_{50\delta}([x,x'])$ where $e'$ is the edge of $[\tilde{x},\tilde{x'}]\subset[x,x']$ which contains $\tilde{x'}$. Thus $z' \in Cone_{150\delta}([x,x'])$ and according to Corollary \ref{denombrement relation dequivalence cones}, the number of equivalence classes $I_{i'}^{n',k',x'}$ is bounded above by $L |Cone_{150\delta}([x,x'])|$.

\end{proof}

We will now prove the reserve counting.

\begin{proposition}\label{denombrement reciproque si z dans x et u et atilde loin}

Under the assumptions of Proposition \ref{zentre aetu facile}, we have that $z \in Cone_{150\delta}([x,x']) $. 
Then the equivalence classes $I_{i'}^{n',k',x'}$, chosen like in Proposition \ref{z entre u et atilde }, appear in at most  $L|Cone_{150\delta}([x,x'])|$ classes $I_{i}^{n,k,x}$ chosen like in Proposition \ref{z entre u et atilde }.

\end{proposition}

\begin{proof}
If $z \in [x,\tilde{x}]$ then $z \in [x,x']$ according to Theorem \ref{formenormaledestriangles} and so $z \in Cone_{150\delta}([x,x'])$. In the rest of proof, let us assume that $z \in [\tilde{x},u]$.\\

If $d(u,z) \ge 8 \delta$, then according to Proposition \ref{conetriangleplusfort}, $z \in Cone_{50\delta}([x,x']) $.\\

For that reason, let us assume that $d(u,z) \le 8 \delta $.\\

Let assume further that $d(u,\tilde{x})\ge 8\delta$. We denote by $e$ the edge of the side $[u,\tilde{x}]$ such that $d(u,e)=8\delta$. Since $z \in [u,\tilde{x}]$, then $d(z,e)\le 8\delta$. Furthermore $z \in [\tilde{a},\tilde{x}]$, then according to Theorem \ref{formenormaledestriangles}, along the side of the triangle $[x,a]$, angles from $z$ to $e$ are bounded above by $100\delta$. Thus $z \in Cone_{100\delta}(e)$. According to Proposition \ref{conetriangleplusfort}, $e\in Cone_{50\delta}([x,x'])$ since $d(u,e)=50\delta$ and then we can deduce that $z \in Cone_{150\delta}([x,x'])$.\\

Let assume now that $ d(u,\tilde{x}) < 8 \delta $. We denote now by $e$ the edge of $[u,\tilde{x}]$, which contains $\tilde{x}$. Since $z \in [u,\tilde{x}]$ and $e$ is an edge of $[ \tilde{a},\tilde{x}]$, angles from $z$ to $e$ are bounded above by $100 \delta$, furthermore $d(e,z)\le 8\delta$ and then $ z \in Cone_{ 100 \delta } (e)$. According to Proposition \ref{conetriangleplusfort} in the triangle $ [ \tilde{x},\tilde{x'},\tilde{a}]$, $e \in Cone_{50\delta}(e') \subset Cone_{50\delta}([x,x'])$ where $e'$ is the edge of $ [ \tilde{x},\tilde{x'}] \subset  [x,x']$ which contains $\tilde{x}$. Thus $z \in Cone_{  150 \delta}([x,x'])$.\\

To conclude the proof of the proposition, we just need to use Corollary \ref{denombrement relation dequivalence cones} and we get the bound $L|Cone_{150\delta}([x,x'])|$.

\end{proof}

\subsection{The case where $z\in [u,\tilde{a}]$ and $\tilde{a}$ is far enough from $z$} \label{subsec : troisième cas}

\begin{proposition} \label{z entre u et atilde }

Let assume that $k > (a,x')_{x}$ and $k \le d(x,\tilde{a})$, i.e. $z\in [u,\tilde{a}]$ and $z\neq u$ , and that $d(x',\tilde{a}) \geq d(x',z)+3\delta$. Then if we set $n'=d(x',a)$, $k'=d(x',z)+3\delta$ and $I_{i'}^{n',k',x'}$ the equivalence class of $a$ for $x'$, $n'$ and $k'$, we have $I_{i'}^{n',k',x'} \subset I_{i}^{n,k,x}$.\\

\end{proposition}

\begin{figure}[!ht]
    \centering
    \includegraphics[scale=0.7]{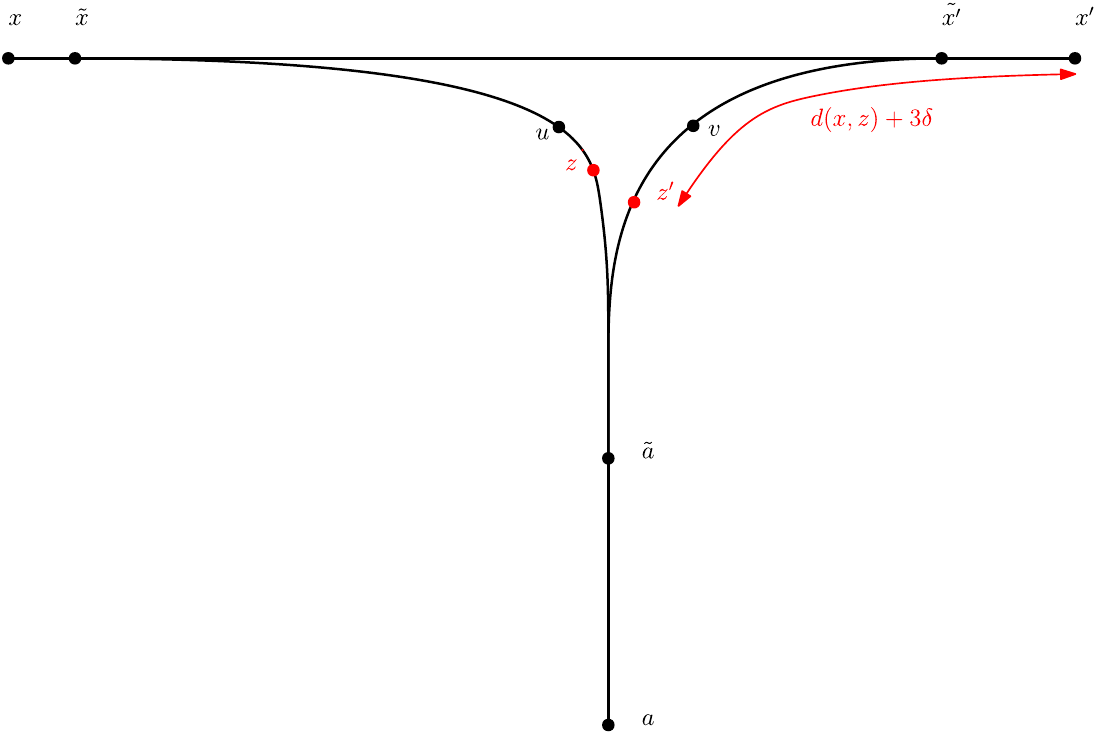}
    \caption{Case where $z$ is between $u$ and $\tilde{a}$ and $\tilde{a}$ is far enough }
    \label{ fig: cas oÃ¹ z est avant u ( premier dessin) }
\end{figure}

\begin{proof}

To prove the inclusion, we have to use Lemma \ref{classedequnpoint} since $B(z,3\delta) \subset B(x',k')$, then $I_{i'}^{n',k',x'} \subset I_{i}^{n,k,x}$.

 \end{proof}

 \begin{proposition} \label{denombrement cas où z dans [u,tilde a] et tilde a loin de z}

Under the assumptions of Proposition \ref{z entre u et atilde }, we denote by $\hat{e}$ an edge such that $\hat{e} \subset I(x,z)$, such that $z \in \hat{e}$ , the number of classes $I_{i'}^{n',k',x'}$ in the decomposition is bounded above by $ L |Cone_{200\delta}(\hat{e})|$, where $L$ is the constant defined in Corollary \ref{denombrement relation dequivalence cones}.

\end{proposition}

\begin{proof}

 To control the number of classes in the decomposition, we will bound the number of $z'$ such that $d(x',z')=k'$ and $z'\in [x',a]$ and conclude thanks to Corollary \ref{denombrement relation dequivalence cones}. Let $e$ be the edge such that  $e \subset [x,z]$ and $z \in e$.\\

 Let assume that $d(v,\tilde{a})> 8\delta$. We denote by $f$ the edge such that $f \subset[v,\tilde{a}]$ such that $d(f,v)=8\delta$. We remark that $d(z',f)< 8\delta$. Since $z'\in [x',\tilde{a}]$ and $f \subset [x',tilde{a}]$, according to Theorem \ref{formenormaledestriangles}, angles on the side of the triangle from $z'$ to $f$ are bounded by $100\delta$, thus $z'\in Cone_{100\delta}(f)$. According to Proposition \ref{conetriangleplusfort}, there exists an edge $f'$ such that $f' \subset [x,a]$ at the same distance from $a$ than $f$ such that $f \in Cone_{50\delta}(f')$. Therefore we remark that $f'$ is an edge of $[u,\tilde{a}]\subset [\tilde{x},\tilde{a}]$ and such that $d(e,f')=d(z,f')\le d(u,f')\le 8 \delta$ since $d(u,f')=d(u,a)-d(a,f')=d(v,a)-d(a,f)=d(v,f)=8\delta$. Thus $d(z',e)\le d(z',f)+d(f,f')+d(f',e) \le 66\delta$. A new use of Theorem \ref{formenormaledestriangles} gives the fact that every angle on this path from $z'$ to $e$ is bounded by $100\delta$ and so we have that $z'\in Cone_{100\delta}(e)$. The edge $e$ depends on the choice of the triangle $[x,x',a]$, therefore it depends on $a$, we need to include $z'$ in a set that does not depend on $a$. Then to conclude, we use Theorem \ref{conetriangleplusfort} in the degenerate triangle with the edges $[x,u]$ and the geodesic between $x$ and $u$ which contains $\hat{e}$ to remark that $e\in Cone_{50\delta}(\hat{e})$ and we conclude that $z' \in Cone_{150\delta}(\hat{e})$. \\

Let assume that $d(v,\tilde{a})\le 8\delta$. We denote $f$ the edge such that $f \subset[\tilde{a},x']$ and $\tilde{a} \in f$. According to Theorem \ref{formenormaledestriangles}, $z'\in Cone_{100\delta}(f)$ since $z'\in [\tilde{a},\tilde{x}]$ and $f$ such that $d(z',f) \le 8\delta$. Applying Theorem \ref{formenormaledestriangles} to the triangle  $  [ \tilde{x},\tilde{x'}, \tilde{a}]$ gives the fact that $f \in Cone_{50\delta}(f')$ where $f'$ is the edge such that $f' \subset [x,\tilde{a}]$ and $ \tilde{a} \in f'$. According to the assumption $d(v,\tilde{a})\le 8\delta$, we have $d(f,u)\le 8\delta$, then $d(f,z)\le 8\delta$. Furthermore $z \in [u,\tilde{a}]$ then angles on the side of the triangle from $f$ to $z$ are bounded by $100 \delta$ and $f\in Cone_{100\delta}(e)$. In the same way as the previous case, we have $z' \in Cone_{150\delta}(e)$ and by application of Proposition \ref{conetriangleplusfort}, we obtain $z' \in Cone_{200\delta}(\hat{e})$. \\

In both cases, we have $z' \in Cone_{200\delta}(\hat{e})$. By application of Proposition \ref{conereldeq}, we can finally prove the proposition and control the number of classes $I_{i'}^{n',k',x'}$ in the decomposition.

 \end{proof}

 We will now do the reverse counting.

 \begin{proposition} \label{denombrement inverse cas z dans u atilde}

Let $a \in I_{i}^{n,k,x}$ under the assumptions of Proposition \ref{z entre u et atilde } and $I_{i'}^{n',k',x'}$ the associated class chosen like in Proposition \ref{z entre u et atilde }. We denote $z'$ the point on $B(x',k')$ at minimal distance of $a$ on the triangle, i.e. the point on the side $[x',a]$ of the triangle such that $d(x',z')=k'$ and $d(z',a)=n'-k'$.\\
Then for any edge $\hat{f} \subset I(z',x')$ such that $z'\in \hat{f}$, we have that $z \in Cone_{201\delta}(\hat{f})$. \\

The equivalence class $I_{i'}^{n',k',x'}$ appears in the decomposition of at most $L|Cone_{201\delta}(\hat{f})|$ classes $I_{i}^{n,k,x}$ chosen like in Proposition \ref{z entre u et atilde }.

\end{proposition}

\begin{proof}
The proof of this fact is kind of a reciprocal of Proposition \ref{denombrement cas où z dans [u,tilde a] et tilde a loin de z}. In Proposition \ref{denombrement cas où z dans [u,tilde a] et tilde a loin de z}, we proved that $z' \in Cone_{150\delta}(e)$ where $e$ is the edge such that $e \in [\tilde{a},z]$ and $z\in e$. We denote by $\bar{f}$ the edge of the triangle such that $\bar{f} \subset [z',\tilde{a}]$ and $z' \in \bar{f}$. From the fact that $z' \in Cone_{150\delta}(e) $, we know that there exists a path from $z'$ to $e$ of length bounded by $150\delta$ and where every angle is bounded by $150\delta$. Since $z\in e$ and $z' \in \bar{f}$, there exists a path from $z$ to $\bar{f}$ of length bounded by $150\delta$ and where every angle is bounded by $150\delta$, therefore $z \in Cone_{150\delta}(\bar{f})$.  \\

We denote by $f$ the edge of the triangle such that $f \subset [z',x']$ and $z' \in f$. Since $z' \in [\tilde{a},\tilde{x'}]$, according to Theorem \ref{formenormaledestriangles}, we have that $\measuredangle_{v}(\bar{f},f) \leq 100\delta$. From this fact, we deduce that $z \in Cone_{151\delta}(f)$.\\

In Proposition \ref{z entre u et atilde }, we proved that $z' \in Cone_{200\delta}(\hat{e})$, where $\hat{e}$ is an edge such that $\hat{e} \subset I(z,x)$ and $z\in \hat{e}$. Therefore, $z \in Cone_{200\delta}(e)$, where $e$ is the edge such that $e \in [\tilde{a},z']$ and $z'\in e$.\\

To conclude, we have to use Proposition \ref{conetriangleplusfort} in the degenerate triangle $([z',x'],\gamma)$, where $\gamma$ is a geodesic from $z'$ to $x'$, which contains $\hat{f}$. Then we know that $f' \in Cone_{50\delta}(\hat{f})$ and thus $z'\in Cone_{201 \delta}(\hat{f})$.\\

We conclude the proof of the Proposition by an application of Corollary \ref{denombrement relation dequivalence cones}.

\end{proof}

\subsection{ The cases where $\tilde{a}$ is too close from $v$ or too close from $z$ } \label{subsec : quatrième cas}

In this subsection, we assume that $z\in [x,u]$ and $d(x',\tilde{a}) < d(x',v) + 8\delta$ or that $z\in [u,\tilde{a}]$, $z\neq u$ and $d(x',\tilde{a}) < d(x',z)+3\delta$. We set $n'=d(x',a)$, $k'=d(x',\tilde{a})$ and $I_{i'}^{n',k',x'}$ the equivalence class of $a$ for $x'$, $n'$ and $k'$. By definition, we have $a \in I_{i'}^{n',k',x'}$, but since $\tilde{a}$ is too close from $z$, we do not have $I_{i'}^{n',k',x'} \subset I_{i}^{n,k,x}$ in general. We will use the following property of points in $I_{i}^{n,k,x}$:\\

\begin{proposition} \label{tildea sur toute geod}

For every $y \in B(x,k)$, $\tilde{a} \in I(a,y)$.

\end{proposition}

\begin{proof}
Let $y \in B(x,k)$. We have that $\tilde{a} \notin I(x,y)$, since :

$$\begin{aligned} 
            d(x,\tilde{a})+d(\tilde{a},y)& > d(x,\tilde{a})\\      
            & > k \\
            &\geq d(x,y).
\end{aligned}$$\\

Then by contraposition with Proposition \ref{angletoutegéod}, we deduce that $ \measuredangle_{\tilde{a}}(x,y) \leq 12\delta$ . Therefore, since $\measuredangle_{\tilde{a}}(a,x) > 50 \delta$ according to Definition \ref{pointloingrandangle} of $\tilde{a}$, we have : 
$$\begin{aligned} 
            \measuredangle_{\tilde{a}}(a,y)
            & \geq \measuredangle_{\tilde{a}}(a,x) - \measuredangle_{\tilde{a}}(x,y)\\      
            & > 50\delta-12\delta \\
            &\geq 38\delta \\
            & > 12\delta.
\end{aligned}$$

Then by application of Proposition \ref{angletoutegéod}, we deduce that $\tilde{a} \in I(a,y)$. \\

\end{proof}

The following proposition about points in $I_{i}^{n,k,x}$  characterize the points of  $I_{i'}^{n',k',x'}$, which belong to $I_{i}^{n,k,x}$.

\begin{proposition} \label{caracterisationpointpartieenx'dansx}
Let $a' \in I_{i'}^{n',k',x'}$. Then $a' \in I_{i}^{n,k,x}$ if and only if for all $y \in B(x,k) $, $\tilde{a} \in I(a',y)$.

\end{proposition}

\begin{proof}
For all $y \in B(x,k) $ :

$$\begin{aligned} 
      d(a,y)      
    & = d(a,\tilde{a})+d(\tilde{a},y) ~ (\tilde{a} \in I(a,y) ~ \text{according to Proposition \ref{tildea sur toute geod}} ) \\
    & = d(a',\tilde{a})+d( \tilde{a},y) ~ ( d(a,\tilde{a})=d(a',\tilde{a}) ~ \text{since} ~ a,a' \in I_{i'}^{n',k',x'} ~ \text{and} ~ k'=d(x',\tilde{a}) ) \\      
    & \geq d(a',y).
\end{aligned}$$

The equality case holds if and only if $\tilde{a} \in I(a',y)$ and we proved the proposition.\\
    
\end{proof}

\begin{corollary} \label{condition suffisante pointdepartieenx' dans partie en x}

Let $a' \in I_{i'}^{n',k',x'}$. If $\measuredangle_{\tilde{a}}(a',x) > 24 \delta$ then $a' \in I_{i}^{n,k,x} $.

\end{corollary}

\begin{proof}
Let $y \in B(x,k)$, we have :

$$\begin{aligned} 
            \measuredangle_{\tilde{a}}(a',y)
            & \geq \measuredangle_{\tilde{a}}(a',x) - \measuredangle_{\tilde{a}}(x,y)\\      
            & > 24\delta-12\delta \\
            & \geq 12\delta . 
\end{aligned}$$

Then by application of Proposition \ref{angletoutegéod}, we deduce that $\tilde{a} \in I(a',y)$. Then by application of Proposition \ref{caracterisationpointpartieenx'dansx} we have that $a' \in I_{i}^{n,k,x}$.

\end{proof}

\begin{proposition} \label{lapartieprivéedecertaines est incluse}

   We set $(I_{l}^{n',k'+1,x'})_{l \in L}$ the equivalence classes of points in $\{ a' \in I_{i'}^{n',k',x'} ~|~  \measuredangle_{\tilde{a}}(x,a') \le 24\delta \} $ with respect to $(n',k'+1,x')$, with the convention that $I_{l}^{n',n'+1,x'}= \emptyset$. Then :

   $$I_{i'}^{n',k',x'} - \bigcupdot\limits_{ l \in L} I_{l}^{n',k'+1,x'} \subseteq I_{i}^{n,k,x} .$$

\end{proposition}

\begin{proof}

According to Corollary \ref{condition suffisante pointdepartieenx' dans partie en x}, we know that $I_{i'}^{n',k',x'} - \{ a' \in I_{i'}^{n',k',x'} ~|~  \measuredangle_{\tilde{a}}(x,a') \le 24\delta \} \subseteq I_{i}^{n,k,x}$. We define $(I_{l}^{n',k'+1,x'})_{l \in L}$ as the equivalence classes of points in $\{ a' \in I_{i'}^{n',k',x'} ~|~  \measuredangle_{\tilde{a}}(x,a') \le 24\delta \} $ with respect to $(n',k'+1,x')$, with the convention that $I_{l}^{n',n'+1,x'}= \emptyset$. Thereby we have that : 

 $$I_{i'}^{n',k',x'} - \bigcupdot\limits_{ l \in L} I_{l}^{n',k'+1,x'} \subseteq I_{i'}^{n',k',x'} - \{ a' \in I_{i'}^{n',k',x'} ~|~  \measuredangle_{\tilde{a}}(x,a') \le 24\delta \} \subseteq I_{i}^{n,k,x} $$
    
\end{proof}

We still have to prove that $a \in I_{i'}^{n',k',x'} - \bigcupdot\limits_{ l \in L} I_{l}^{n',k'+1,x'}  $
and to enumerate precisely the number of classes, which appear in the decomposition.
The following property will be useful to prove both points.

\begin{proposition}\label{z'' est à angle borné de atilde }
Let $a' \in I_{i'}^{n',k',x'}$ such that $\measuredangle_{\tilde{a}}(x,a') \le 24\delta $. We denote by $I_{l}^{n',k'+1,x'}$ the equivalence class of $a'$ with respect to $(n',k'+1,x')$. There exists a point $z''$ such that $d(x',z'')=k'+1$,  $d(z'',a')=n'-(k'+1)$ and $\tilde{a} \in I(x',z")$, in other words $z''$ is a vertex of $B(x',k'+1)$ at minimal distance of $a'$ such that $\tilde{a} \in I(x',z")$.\\
Furthermore, for all such $z''$, we have that $\measuredangle_{\tilde{a}}(x,z'') \le 24\delta $.
    
\end{proposition}

\begin{proof}

By definition of the class $I_{i'}^{n',k',x'}$, we know that $a \in I_{i'}^{n',k',x'}$ and $k'=d(x',\tilde{a})$ and by Definition \ref{pointloingrandangle} of the point $\tilde{a}$, we know that $\tilde{a} \in I(a,x')$. Therefore, $\tilde{a}$ is a point of $B(x',k')$ at minimal distance of the class $I_{i'}^{n',k',x'}$, then $\tilde{a} \in I(a',x')$. Thus for every geodesic between $x'$ to $a'$ going through $\tilde{a}$, a point on this geodesic at distance $k'+1$ from $a'$ is a point $z''$ as in the definition.\\

To prove the second point, we come back to the definition of angle. We take an edge $e_{1}$  with $\tilde{a} \in e_{1}$ on a geodesic between $\tilde{a}$ and $x$ and an edge $e_{2}$ with $\tilde{a} \in e_{2}$ on a geodesic between $\tilde{a}$ and $z''$. Thanks to the first point the geodesic between $\tilde{a}$ and $z''$ could be extended to a geodesic between $\tilde{a}$ and $a'$. Since $\measuredangle_{\tilde{a}}(x,a') \le 24\delta $, one deduces that $\measuredangle_{\tilde{a}}(e_{1},e_{2}) \le 24\delta $ and then $\measuredangle_{\tilde{a}}(x,z'') \le 24\delta $ by definition of angles.

\end{proof}

\begin{corollary}
$a \notin \bigcupdot\limits_{ l \in L} I_{l}^{n',k'+1,x'}$.

\end{corollary}

\begin{proof}
 We know that $\measuredangle_{\tilde{a}}(x,a) > 50\delta $ by definition of $\tilde{a}$.  Thus there exists an edge $e_{1}$ with $\tilde{a}\in e_{1}$ on a geodesic between $\tilde{a}$ and $x$ and an edge $e_{2}$ with $\tilde{a}\in e_{2}$ on a geodesic between $\tilde{a}$ and $a$ such that $\measuredangle_{\tilde{a}}(e_{1},e_{2}) > 50\delta $. Let us denote $w$ the other vertex of the edge $e_{2}$. Since $\tilde{a} \in I(a,x')$ and $w \in I(a,\tilde{a})$, we can deduce that $w \in I(a,x')$, $\tilde{a} \in I(w,x')$ and that $d(x',w)=d(x',\tilde{a})+d(\tilde{a},w)=k'+1$,  $d(w,a)=n'-(k'+1)$. Thus the vertex $w$ satisfies the same assumptions as the vertex $z''$ in Proposition \ref{z'' est à angle borné de atilde } such that $\measuredangle_{\tilde{a}}(x,w) > 50\delta $, so one deduces that $a \notin  \bigcupdot\limits_{ l \in L} I_{l}^{n',k'+1,x'}$.

\end{proof}

Now we can count the number of classes, which appear in the decomposition.

We will now distinguish two cases :

\begin{itemize}
    \item the case where $z$ is between $x$ and $u$, and $\tilde{a}$ is too close from $v$ ;
    \item the case where $z$ is between $u$ and $\tilde{a}$, and $z$ is too close from $\tilde{a}$.
\end{itemize}

\subsubsection{The case where $z$ is between $x$ and $u$, and $\tilde{a}$ is too close from $v$}

In this paragraph, we will assume that $z\in [x,u]$ and that $d(x',\tilde{a}) < d(x',v) + 8\delta$.

\begin{proposition}\label{denombrement partie quand atilde est trop proche et z dand [x,u]}

 The number of equivalence classes $I_{i'}^{n',k',x'}$ of $a$ for $n'=d(x',a)$, $k'=d(x',\tilde{a})$ , which appear positively in the decomposition, is bounded above by $ L |Cone_{150\delta}([x,x'])|$, where $L$ is the constant defined in Corollary \ref{denombrement relation dequivalence cones}. \\
 The number of equivalence classes $I_{l}^{n',k'+1,x'}$ as in Proposition \ref{z'' est à angle borné de atilde }, which appear negatively in the decomposition, is bounded above by $ L |Cone_{224\delta}([x,x'])|$, where $L$ is the constant defined in Corollary \ref{denombrement relation dequivalence cones}.

\end{proposition}

\begin{figure}[!ht]
    \centering
    \includegraphics[scale=0.7]{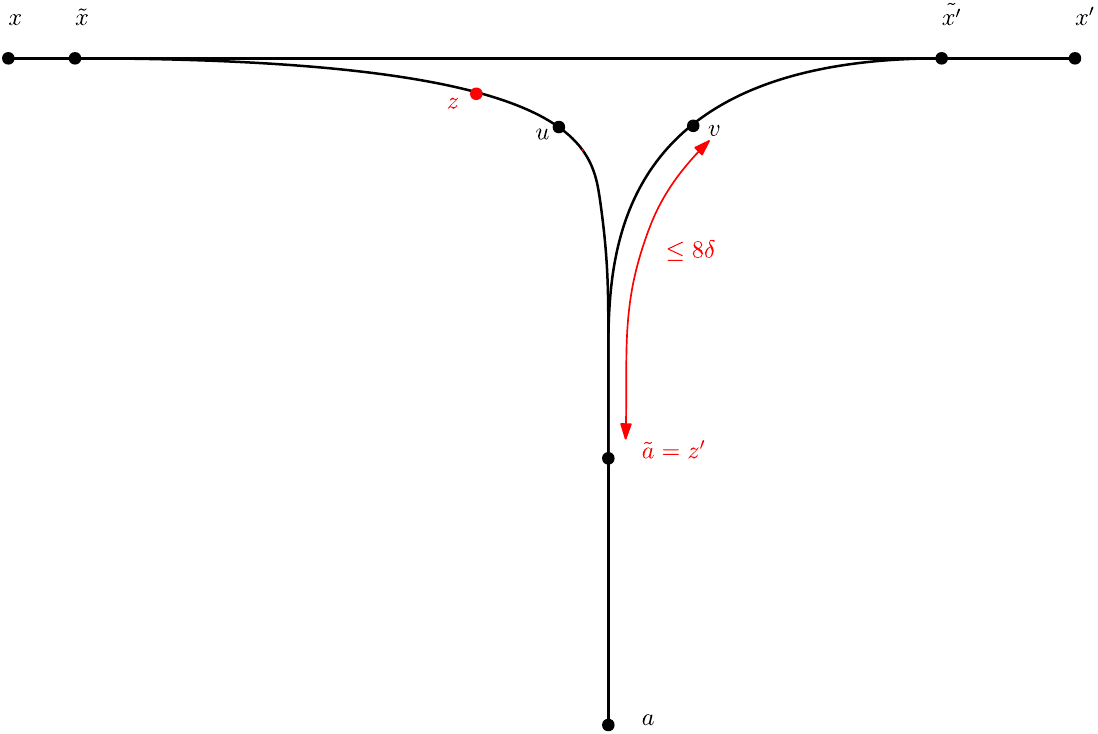}
    \caption{Case where $z$ is between $x$ and $u$, and $\tilde{a}$ is too close from $v$ }
    \label{ fig: cas oÃ¹ z est avant u ( premier dessin) }
\end{figure}

\begin{proof}

Like in the other cases we will bound the number of $z'$ such that $d(x',z')=k'$ and $z'\in [x',a]$ and conclude thanks to Corollary \ref{denombrement relation dequivalence cones}. For the classes, which appears positively, we have in this case that $z'=\tilde{a}$.\\

Let assume that $d(v,\tilde{x}')>8 \delta$. We denote $e$ the edge such that $e \subset [v,x']$ such that $d(v,e)=8\delta$. By assumption, we have that $d(\tilde{a},e)\le 16\delta$. Furthermore, since $e$ is an edge of $[\tilde{a},\tilde{x'}]$ from $\tilde{a}$ to $e$ angles are bounded by $100\delta$ according to Theorem \ref{formenormaledestriangles}, then $\tilde{a} \in Cone_{100\delta}(e)$. According to Proposition \ref{conetriangleplusfort}, as $d(v,e) \ge 8\delta$, there exists an edge $e'$ of the side $[x,x']$ such that $e\in Cone_{50\delta}(e')$ and we can deduce that $\tilde{a} \in Cone_{150\delta}(e') \subset Cone_{150\delta}([x,x'])$.\\

Let assume that $d(v,\tilde{x}')\le 8\delta$. Now we denote $e$ the edge such that $e \subset [v,\tilde{x'}]$ and $\tilde{x'} \in e$. Since $e$ is an edge of $[\tilde{a},\tilde{x'}]$, angles from $\tilde{a}$ to $e$ along the triangle are bounded by $100\delta$, thus $\tilde{a} \in Cone_{100\delta}(e)$. According to Proposition \ref{conetriangleplusfort} in the triangle $  [ \tilde{x},\tilde{x'} , \tilde{a}] $, we know that $e \in Cone_{50\delta}(e')\subset Cone_{50\delta}([x,x'])$, where $e'$ is the edge such that $e' \subset [\tilde{x},\tilde{x'}]\subset[x,x']$ and $e' \in \tilde{x}' $, thus $\tilde{a} \in Cone_{150\delta}([x,x'])$.\\

In both cases $\tilde{a} \in Cone_{150\delta}([x,x'])$ 
, according to Corollary \ref{denombrement relation dequivalence cones}, the number of equivalence classes $I_{i'}^{n',k',x'}$ is bounded above by $L |Cone_{150\delta}([x,x'])|$.\\

To control the number of classes which appear negatively in the decomposition, we will control the number of $z''$ as in Proposition \ref{z'' est à angle borné de atilde }. For such $z''$, we know thanks to this proposition that $d(\tilde{a},z'')=1$ and that $\measuredangle_{\tilde{a}}(x,z'') \le 24\delta $. Then $z'' \in Cone_{24\delta}(f)$ where $f$ is the edge such that $f \subset [x,\tilde{a}]$ and $\tilde{a} \in f$. According to Proposition \ref{conetriangleplusfort}, $z'' \in Cone_{74\delta}(f')$, where $f'$ is the edge such $f' \subset [x',\tilde{a}]$ and $\tilde{a} \in f'$. Then, thanks to the beginning of the proof, one can deduce that $z'' \in Cone_{224\delta}([x,x'])$. This conclude the proof of the proposition.

\end{proof}

We will do now the reverse counting in this case.

\begin{proposition} \label{denombrement reciproque cas atilde trop pret 1}
Under the assumptions of this paragraph, we have that $z \in Cone_{150 \delta}([x,x']) $.

Then the classes $I_{i'}^{n',k',x'}$ and $(I_{l}^{n',k'+1,x'})_{l \in L}$ chosen like in Proposition \ref{lapartieprivéedecertaines est incluse} appear in at most $L |Cone_{150 \delta}([x,x'])|$ decompositions of equivalence classes $I_{i}^{n,k,x}$ chosen under the assumptions of this paragraph.

\end{proposition}

\begin{proof}
The proof of the fact that $z \in Cone_{50\delta}([x,x'])$ is in fact exactly the same as the proof of Proposition \ref{denombrement reciproque si z dans x et u et atilde loin}.\\
If $z \in [x,\tilde{x}]$ then $z \in [x,x']$ according to Theorem \ref{formenormaledestriangles} and so $z \in Cone_{150\delta}([x,x'])$. In the rest of proof, let us assume that $z \in [\tilde{x},u]$.\\

If $d(u,z) \ge 8 \delta$, then according to Proposition \ref{conetriangleplusfort}, $z \in Cone_{50\delta}([x,x']) $.\\

For that reason, let us assume that $d(u,z) \le 8 \delta $.\\

Let assume further that $d(u,\tilde{x})\ge 8\delta$. We denote by $e$ the edge of the side $[u,\tilde{x}]$ such that $d(u,e)=8\delta$. Since $z \in [u,\tilde{x}]$, then $d(z,e)\le 8\delta$. Furthermore $z \in [\tilde{a},\tilde{x}]$, then according to Theorem \ref{formenormaledestriangles}, along the side of the triangle $[x,a]$, angles from $z$ to $e$ are bounded above by $100\delta$. Thus $z \in Cone_{100\delta}(e)$. According to Proposition \ref{conetriangleplusfort}, $e\in Cone_{50\delta}([x,x'])$ since $d(u,e)=50\delta$ and then we can deduce that $z \in Cone_{150\delta}([x,x'])$.\\

Let assume now that $ d(u,\tilde{x}) < 8 \delta $. We denote now by $e$ the edge of $[u,\tilde{x}]$, which contains $\tilde{x}$. Since $z \in [u,\tilde{x}]$ and $e$ is an edge of $[ \tilde{a},\tilde{x}]$, angles from $z$ to $e$ are bounded above by $100 \delta$, furthermore $d(e,z)\le 8\delta$ and then $ z \in Cone_{ 100 \delta } (e)$. According to Proposition \ref{conetriangleplusfort} in the triangle $ [ \tilde{x},\tilde{x'},\tilde{a}]$, $e \in Cone_{50\delta}(e') \subset Cone_{50\delta}([x,x'])$ where $e'$ is the edge of $ [ \tilde{x},\tilde{x'}] \subset  [x,x']$ which contains $\tilde{x}$. Thus $z \in Cone_{  150 \delta}([x,x'])$.\\

To conclude the proof of the proposition, we just need to use Corollary \ref{denombrement relation dequivalence cones} and we get the bound $L|Cone_{150\delta}([x,x'])|$.

\end{proof}
    
\newpage

\subsubsection{The case where $z$ is between $u$ and $\tilde{a}$, and $z$ is too close from $\tilde{a}$  }

In this paragraph, we will assume that $z\in [u,\tilde{a}]$, $z\neq u$ and that $d(x',\tilde{a}) < d(x',z)+3\delta$.

\begin{figure}[!ht]
    \centering
    \includegraphics[scale=0.7]{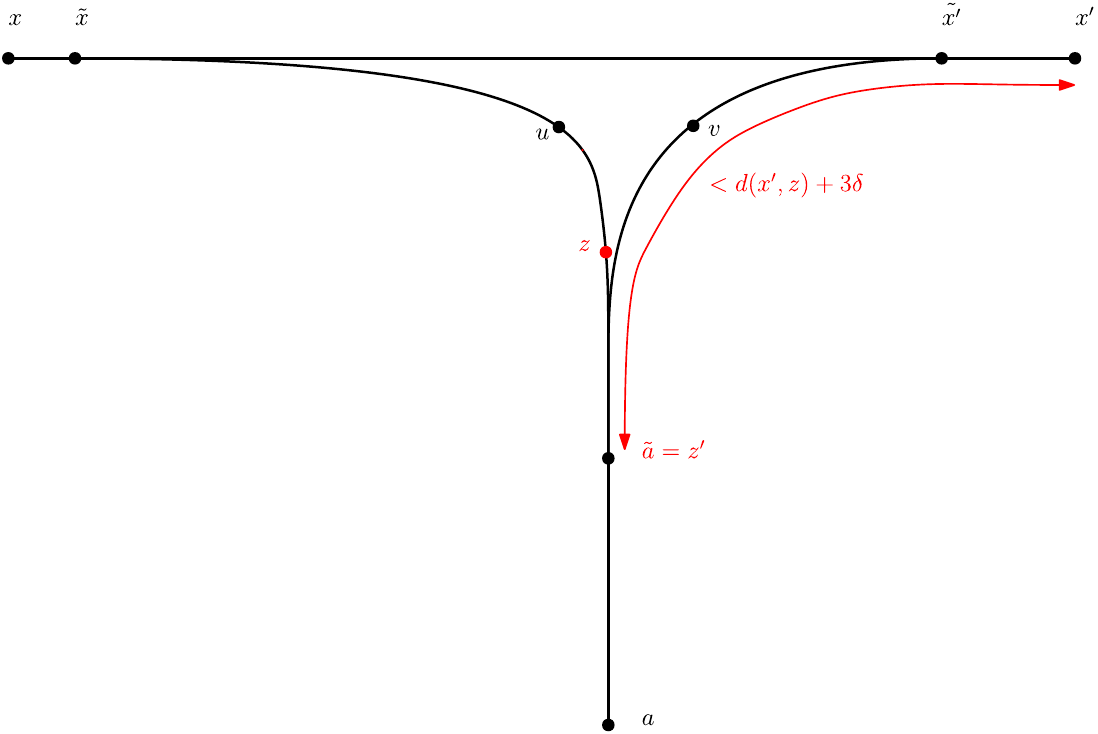}
    \caption{Case where $z$ is between $u$ and $\tilde{a}$, and $z$ is too close from $\tilde{a}$ }
    \label{ fig: cas oÃ¹ z est avant u ( premier dessin) }
\end{figure}

Under the assumptions of this paragraph, $z$ is at bounded distance of $\tilde{a}$. 
We will need this fact to count the number of equivalence classes in that case.

\begin{lem} \label{grace à Martin z est proche de tildea }

Under the assumptions of this paragraph, we have that $d(\tilde{a},z)\le 23\delta$.

\end{lem}

\begin{proof}

We will apply Lemma \ref{lemme de Bridson} to the geodesics $c=[a,x]$, $c'=[a,x']$. Since $d(u,v)\le 4\delta$, we know that $d(c(t_{0}),Im(c')) \le 2\delta$, with $t_{0}=(x|x')_{a}$. Then with Lemma \ref{lemme de Bridson} and the fact that triangles are $\delta$-thin, we know that $d(c(t),c'(t))\le 2\delta$ for all $t \le (x|x')_{a}-6\delta$.\\

Then for $z \in [u,\tilde{a}]$ if $d(z,u)\ge 6\delta$, according to Lemma \ref{lemme de Bridson}, there exists $y \in [\tilde{a},v]$ such that $d(z,y) \le 2\delta$. Then $d(z,x') \le 2\delta + d(y,x') $ and $d(\tilde{a},z) \le  d(\tilde{a},y) + 2 \delta$, therefore :

$$ d(\tilde{a},z)+d(z,x') \le d(\tilde{a},y)+d(y,x')+4\delta \le d(\tilde{a},x')+4\delta .$$

Since $y \in [\tilde{a},x']$. From this we deduce that : 

$$ d(\tilde{a},z) \le d(\tilde{a},x')+64\delta - d(z,x') \le 7 \delta$$

with the assumption that $d(x',\tilde{a}) < d(x',z)+3\delta$.\\

If $d(z,u) < 6\delta$, then $d(z,v) \le 10 \delta$. Therefore :

$$ d(\tilde{a},z)+d(z,x') \le d(\tilde{a},v)+d(v,x')+20\delta \le d(\tilde{a},x')+20\delta .$$

Then, we deduce that :

$$ d(\tilde{a},z) \le d(\tilde{a},x')+20\delta - d(z,x') \le 23 \delta$$

with the assumption that $d(x',\tilde{a}) < d(x',z)+3\delta$.\\

In both case, $d(z,\tilde{a})\le 23\delta$, so we proved the lemma.
    
\end{proof}

\begin{proposition}\label{denombrement partie quand atilde est trop proche et z dand [u,atilde]}

Let $\hat{e}$ be an edge such that $\hat{e} \subset I(x,z)$ and $z \in \hat{e}$.

 The number of equivalence classes $I_{i'}^{n',k',x'}$ of $a$ for $n'=d(x',a)$, $k'=d(x',\tilde{a})$ , which appear positively in the decomposition, is bounded above by $ L |Cone_{150\delta}(\hat{e})|$, where $L$ is the constant defined in Corollary \ref{denombrement relation dequivalence cones}. \\
 The number of equivalence classes $I_{l}^{n',k'+1,x'}$ as in Proposition \ref{z'' est à angle borné de atilde }, which appear negatively in the decomposition, could be bounded above by $ L |Cone_{174\delta}(\hat{e})|$, where $L$ is the constant defined in Corollary \ref{denombrement relation dequivalence cones}.

\end{proposition}

\begin{proof}

Like the other cases we will bound the number of $z'$ such that $d(x',z')=k'$ and $z'\in [x',a]$ and conclude thanks to Corollary \ref{denombrement relation dequivalence cones}. For the classes which appears positively, we have in this case that $z'=\tilde{a}$.
Let $e$ be the edge such $e \subset [x,z]$ such that $z \in e$.\\

Since $z \in [\tilde{x},\tilde{a}]$ according to Theorem \ref{formenormaledestriangles} angles from $\tilde{a}$ to $z$ are bounded above by $100\delta$. According to Lemma \ref{grace à Martin z est proche de tildea }, we know that $d(z,\tilde{a})\le 23\delta$. Therefore, $\tilde{a} \in Cone_{100\delta}(e)$. An application of Proposition \ref{conetriangleplusfort} in the degenerate triangle with sides $[x,z]$ and the geodesic from $x$ to $z$ which contains $\hat{e}$ gives the fact that $ z \in Cone_{150\delta}(\hat{e})$.\\

To control the number of classes, which appear negatively in the decomposition, we will control the number of $z''$ as in Proposition \ref{z'' est à angle borné de atilde }. For such $z''$, we know thanks to this proposition that $d(\tilde{a},z'')=1$ and that $\measuredangle_{\tilde{a}}(x,z'') \le 24\delta $. Then $z'' \in Cone_{24\delta}(f)$ where $f$ is the edge such that $f \subset [x,\tilde{a}]$ and $\tilde{a} \in f$. Then, thanks to the beginning of the proof, one can deduce that $z'' \in Cone_{174\delta}(\hat{e})$. This conclude the proof of the proposition.

\end{proof}

We will do now the reverse counting in this case.

\begin{proposition} \label{denombrement reciproque cas a tilde trop pret 2}
Let $a \in I_{i}^{n,k,x}$ under the assumptions of Proposition \ref{z entre u et atilde } and $I_{i'}^{n',k',x'}$, $(I_{l}^{n',k'+1,x'})_{l \in L}$ the associated classes chosen like in Proposition \ref{z entre u et atilde } and Proposition \ref{lapartieprivéedecertaines est incluse}.\\

Then for any edge $\hat{f} \subset I(\tilde{a},x')$ such that $\tilde{a}\in \hat{f}$, we have that $z \in Cone_{200\delta}(\hat{f})$. \\

The equivalence classes $I_{i'}^{n',k',x'}$ and $(I_{l}^{n',k'+1,x'})_{l \in L}$ appear in the decomposition of at most $L|Cone_{200\delta}(\hat{f})|$ classes $I_{i}^{n,k,x}$ chosen like in Proposition \ref{z entre u et atilde }.

\end{proposition}

\begin{proof}

Let $e$ be the edge such that $e \subset [x,\tilde{a}]$ and $\tilde{a} \in e$. By Lemma \ref{grace à Martin z est proche de tildea } and Theorem \ref{formenormaledestriangles}, $z \in Cone_{100\delta}(e)$.\\

Let $f$ be the edge such that $f \in [\tilde{a},x']$ and  $\tilde{a} \in f$, according to Proposition \ref{conetriangleplusfort}, $e \in Cone_{50\delta}(f)$, therefore $z \in Cone_{150\delta}(e)$. Another application of Proposition \ref{conetriangleplusfort}, in the degenerate triangle $([\tilde{a},x'],\gamma)$, where $\gamma$ is a geodesic from $z'$ to $x'$, which contains $f$. We conclude that $z \in Cone_{200\delta}(\hat{f}) $.

\end{proof}

\subsection{Conclusion of the proof} \label{subsec : conclusion}

We are now ready to conclude the proof of Theorem \ref{theorem:main result}. The following lemmas allow to prove that the decomposition of $I_{i}^{n,k,x}$ in term of classes $I_{i'}^{n',k',x'}$ is disjoint.

\begin{lem} \label{casrelou de inclus ou disjoint}

Let $a, a' \in I_{i}^{n,k,x}$. Let  assume that $a$ and $a'$ are both in case of subsection \ref{subsec : quatrième cas}, in other words $\tilde{a}$ is too close from $v$ or too close from $z$ and the same property is true for $a'$. Let assume further that the equivalence classes, as in Proposition \ref{caracterisationpointpartieenx'dansx}, $I_{i'}^{n',k',x'}$, $I_{i''}^{n',k'',x'}$ for $k'=d(x',\tilde{a})$ and $a$ and for $k''=d(x'',\tilde{a'})$ and $a'$ respectively, are equal.

Then $\tilde{a}=\tilde{a'}$, therefore the sets of equivalence classes defined in Proposition \ref{lapartieprivéedecertaines est incluse} that we have to remove to $I_{i'}^{n',k',x'}$ with respect to $a$ and to $a'$ are the same. 

\end{lem}

\begin{proof}

By definition of the class $I_{i''}^{n',k'',x'}$, and since $I_{i'}^{n',k',x'}=I_{i''}^{n',k'',x'}$, $\tilde{a'}$ is a vertex in $B(x',k'')$
at minimal distance of $a$ (i.e. $n-k$). Then, we can deduce that $\tilde{a'} \in I(a,x')$, there exists a geodesic $\gamma$ from $x'$ to $a$, which goes through $\tilde{a'}$.\\

By Definition \ref{pointloingrandangle} of $\tilde{a}$, $\tilde{a} \in \gamma$. Since $I_{i'}^{n',k',x'}=I_{i''}^{n',k'',x'}$, $d(x',\tilde{a})=k'=k''=d(x',\tilde{a'})$. Then, we deduce that $\tilde{a}=\tilde{a'}$.\\

In Proposition \ref{lapartieprivéedecertaines est incluse}, the classes we remove from $I_{i'}^{n',k',x'}$ are defined by $I_{i'}^{n',k',x'}$
 and $\tilde{a}$ only. Then since $\tilde{a}=\tilde{a'}$, the classes we have to remove to $I_{i'}^{n',k',x'}$ with respect to $a$ and to $a'$ are the same.

\end{proof}

\begin{lem}\label{disjoint ou inclus}

Let $(n',k',i'), (n'',k'',i'') \in U_{x'} $, we consider the classes $I_{i'}^{n',k',x'}$ and $I_{i''}^{n'',k'',x'}$. Let $(I_{l}^{n',k'+1,x'})_{l \in L}$, $(I_{p}^{n'',k''+1,x'})_{p \in P}$ be sets of equivalence classes, chosen like in Proposition \ref{lapartieprivéedecertaines est incluse}, included in $I_{i'}^{n',k',x'}$ and $I_{i''}^{n'',k'',x'}$ respectively. We have the following properties :   

\begin{itemize}
    \item the classes $I_{i'}^{n',k',x'}$ and $I_{i''}^{n'',k'',x'}$ are either disjoint or one is included in the other;
    \item the sets $I_{i'}^{n',k',x'}$ and $I_{i''}^{n'',k'',x'} - \bigcupdot\limits_{ p \in P} I_{p}^{n'',k''+1,x'}$ are either disjoint or one is included in the other;
    \item the sets $I_{i'}^{n',k',x'}- \bigcupdot\limits_{ l \in L} I_{l}^{n'',k''+1,x'}$ and $I_{i''}^{n'',k'',x'} - \bigcupdot\limits_{ p \in P} I_{p}^{n'',k''+1,x'}$ are either disjoint or one is included in the other.
\end{itemize}

\end{lem}

\begin{proof}
The first point is proved in Lafforgue's paper \cite{Lafforgue}. Let us recall it for the reader's convenience. Let us assume that $I_{i'}^{n',k',x'} \cap I_{i''}^{n'',k'',x'} \neq \emptyset $, we want to prove that $I_{i'}^{n',k',x'}\subset I_{i''}^{n'',k'',x'}$ or either $I_{i''}^{n'',k'',x'}\subset I_{i'}^{n',k',x'}$. Let $a \in I_{i'}^{n',k',x'} \cap I_{i''}^{n'',k'',x'} $  then by definition of the classes $n'=d(a,x')=n''$ so $n'=n''$. Then assume that $k' \geq k''$ and let $a' \in I_{i'}^{n',k',x'} $. By definition of the classes , for all $u \in B(x',k')$, $d(a',u)=d(a,u)$. It also true in particular for all $u \in B(x',k'')$ since $B(x',k'') \subset B(x',k')$ and then $a' \in I_{i'}^{n',k'',x'}  $. We deduce that $I_{i'}^{n',k',x'} \subset I_{i''}^{n',k'',x'}$ and the inclusion is strict when $k' \le k''$. In case where $k' \leq k''$, the reserve inclusion is true.\\

To prove the second point, let us assume that $I_{i'}^{n',k',x'} \cap (I_{i''}^{n'',k'',x'} - \bigcupdot\limits_{ p \in P} I_{p}^{n'',k''+1,x'}) \neq \emptyset $. Let $a \in I_{i'}^{n',k',x'} \cap (I_{i''}^{n'',k'',x'} - \bigcupdot\limits_{ p \in P} I_{p}^{n'',k''+1,x'}) $. In particular $a \in I_{i'}^{n',k',x'} \cap I_{i''}^{n'',k'',x'}$, according to the first point then $I_{i'}^{n',k',x'} \subset I_{i''}^{n'',k'',x'} $ or $ I_{i''}^{n'',k'',x'} \subset I_{i'}^{n',k',x'} $. If the second inclusion is true then in particular $(I_{i''}^{n'',k'',x'} - \bigcupdot\limits_{ p \in P} I_{p}^{n'',k''+1,x'}) \subset I_{i'}^{n',k',x'}$. Therefore, we assume that $I_{i'}^{n',k',x'} \subsetneq I_{i''}^{n'',k'',x'} $. According to the proof of the first point then $k'> k''$ . In particular $k' \geq k''+1$ then for every $p \in P$, $I_{i'}^{n',k',x'}$ is either disjoint or included in $I_{p}^{n'',k''+1,x'}$. Since $I_{i'}^{n',k',x'} \cap (I_{i''}^{n'',k'',x'} - \bigcupdot\limits_{ p \in P} I_{p}^{n'',k''+1,x'}) \neq \emptyset $, $I_{i'}^{n',k',x'}$ could not be included in some $I_{p}^{n'',k''+1,x'}$, then $I_{i'}^{n',k',x'}$ is disjoint of every $I_{p}^{n'',k''+1,x'}$ and then $I_{i'}^{n',k',x'} \subset (I_{i''}^{n'',k'',x'} - \bigcupdot\limits_{ p \in P} I_{p}^{n'',k''+1,x'}) $.\\

To prove the third point, let us assume that $(I_{i'}^{n',k',x'} - \bigcupdot\limits_{ l \in L} I_{l}^{n'',k''+1,x'}) \cap (I_{i''}^{n'',k'',x'} - \bigcupdot\limits_{ p \in P} I_{p}^{n'',k''+1,x'}) \neq \emptyset$. To begin, we can remark that if $I_{i'}^{n',k',x'}=I_{i''}^{n'',k'',x'}$, then according to Lemma \ref{casrelou de inclus ou disjoint}, we have that $ (I_{i'}^{n',k',x'} - \bigcupdot\limits_{ l \in L} I_{l}^{n'',k''+1,x'}) = (I_{i''}^{n'',k'',x'} - \bigcupdot\limits_{ p \in P} I_{p}^{n'',k''+1,x'}) $. Let us assume that $I_{i'}^{n',k',x'} \neq I_{i''}^{n'',k'',x'}$. Let $a \in (I_{i'}^{n',k',x'} - \bigcupdot\limits_{ l \in L} I_{l}^{n'',k''+1,x'}) \cap (I_{i''}^{n'',k'',x'} - \bigcupdot\limits_{ p \in P} I_{p}^{n'',k''+1,x'})$, then $a \in I_{i'}^{n',k',x'} \cap (I_{i''}^{n'',k'',x'} - \bigcupdot\limits_{ p \in P} I_{p}^{n'',k''+1,x'}) $. According to the second point either $I_{i'}^{n',k',x'} \subset (I_{i''}^{n'',k'',x'} - \bigcupdot\limits_{ p \in P} I_{p}^{n'',k''+1,x'}) $ and we deduce that $ (I_{i'}^{n',k',x'} - \bigcupdot\limits_{ l \in L} I_{l}^{n'',k''+1,x'}) \subset (I_{i''}^{n'',k'',x'} - \bigcupdot\limits_{ p \in P} I_{p}^{n'',k''+1,x'}) $, or $(I_{i''}^{n'',k'',x'} - \bigcupdot\limits_{ p \in P} I_{p}^{n'',k''+1,x'}) \subset I_{i'}^{n',k',x'} $. In this case, since $I_{i''}^{n'',k'',x'}$ and $I_{i'}^{n',k',x'} $ are assumed to be different, we remark that $k'<k''$ and in particular $k'+1 \leq k''$. Thereby $I_{i''}^{n'',k'',x'}$ is either disjoint or included in some $I_{p}^{n'',k''+1,x'}$. By assumption, it could not be included  in some $I_{p}^{n'',k''+1,x'}$, so we deduce that $(I_{i''}^{n'',k'',x'}  - \bigcupdot\limits_{ p \in P} I_{p}^{n'',k''+1,x'}) \subset (I_{i'}^{n',k',x'} - \bigcupdot\limits_{ l \in L} I_{l}^{n'',k''+1,x'}) $.

\end{proof}

We can prove Proposition \ref{propriété clé décomposition}, which summarizes the decompositions and gives a bound, in terms of $d(x,x')$, on the cardinality of the classes $I_{i'}^{n',k',x'}$ we have to use.

\begin{proof} of Proposition \ref{propriété clé décomposition}

Let consider $I_{i}^{n,k,x}$ for some $(n,k,i) \in U_{x}$.
The decomposition follows from Propositions \ref{ cas z  dans [tilde a, a]}, \ref{z entre u et atilde }, \ref{lapartieprivéedecertaines est incluse}. The fact that the decomposition is disjoint comes from Lemma \ref{disjoint ou inclus}.\\

According to Propositions \ref{ cas z  dans [tilde a, a]}, \ref{denombrement cas où z dans [u,tilde a] et tilde a loin de z}, \ref{denombrement partie quand atilde est trop proche et z dand [x,u]}, \ref{denombrement partie quand atilde est trop proche et z dand [u,atilde]}, the number of classes $I_{i'}^{n',k',x'}$, which appear positively in the decomposition is bounded above by : $$ L |Cone_{150\delta}([x,x'])| + L |Cone_{200\delta}(\hat{e})| .$$

In the same way, the number of classes, which appear negatively, is bounded above by :

$$ L |Cone_{224\delta}([x,x'])| + L |Cone_{174\delta}(\hat{e})| . $$

According to Propositions \ref{denombrement reciproque si z dans x et u et atilde loin}, \ref{denombrement inverse cas z dans u atilde}, \ref{denombrement reciproque cas atilde trop pret 1}, \ref{denombrement reciproque cas a tilde trop pret 2}, each class $I_{i'}^{n',k',x'}$ appears in the decomposition of at most $ L |Cone_{150\delta}([x,x'])| + L |Cone_{201\delta}(\hat{f})|  $ classes $I_{i}^{n,k,x}$.\\

By definition of $Cone_{224\delta}([x,x'])$, we know that :

$$ |Cone_{224\delta}([x,x'])| \leq  \sum_{e \in [x,x']} |Cone_{224\delta}(e)| \leq d(x,x') |Cone_{224\delta}(\hat{e})|.$$

Since the cardinality of the cones does not depend on the edge by Proposition \ref{cardinal cones}. Then, if we set $K=L|Cone_{224\delta}(\hat{e})|$, the proposition is proved.

\end{proof}

We can now compare the norms based at the point $x$ and based at the point point $x'=g.x$.

\begin{proposition} \label{norme de la matrice version 2}

Let $A$ be the matrix from $\ell^{2} (U_{x'}) $ to $\ell^{2} (U_{x})$ with coefficient $\frac{n+1}{n'+1}$ if $I_{i'}^{n',k',x'}$ appears positively in the decomposition of $I_{i}^{n,k,x}$, $-\frac{n+1}{n'+1}$ if $I_{i'}^{n',k',x'}$ appears negatively in the decomposition of $I_{i}^{n,k,x}$ and $0$ otherwise. Then $A \circ \Theta_{x'}=\Theta_{x}$ and furthermore :

$$ \vert\vert A\lvert\lvert \le (Kd(x,x')+K) (d(x,x')+1) . $$

\end{proposition}

\begin{proof}
The first property is a direct corollary of the decompositions described in Proposition \ref{propriété clé décomposition}. To control the norm of the matrix $A$, we remark that according to Proposition \ref{propriété clé décomposition}, in each line of $A$ and in each column of $A$ there are at most $Kd(x,x')+K$ non-zero coefficients.

Furthermore, let $I_{i'}^{n',k',x'}$ be a class in the decomposition of $I_{i}^{n,k,x}$ and $a \in I_{i}^{n,k,x}\cap I_{i'}^{n',k',x'} $. Then :

$$ \vert n-n' \lvert = \vert d(x,a)-d(x',a) \lvert \le d(x,x'). $$

From this inequality, we get a bound on the coefficents of $A$, indeed $\frac{n+1}{n'+1} \leq d(x,x')+1$. Therefore, we get the following inequality on the norm of $A$ : 
$$\vert\vert A\lvert\lvert \le (Kd(x,x')+K) (d(x,x')+1) . $$
\end{proof}

To conclude, we will show that the map $\sum_{a} f(a)e_{a} \mapsto \sum_{a} f(a)$ is a continous linear form.

\begin{proposition}\label{formelinéairecontinue2}

The map :

$$\phi:\begin{array}{ccc}
  H  &  \to & \mathbb{C} \\
   \sum_{a} f(a)e_{a}  & \mapsto & \sum_{a} f(a)
\end{array} $$

is a continuous linear form.

\end{proposition}

The proof is similar to the proof of Proposition \ref{formelinéairecontinue}.

To summarize, we built a continuous representation from $G$ on an Hilbert space which is sub-exponential.
Furthermore, this representation preserves the linear form $\phi$, which is continous according to Proposition \ref{formelinéairecontinue2}. As $G$ has the strong Property $(T)$, according to Proposition \ref{Proposition : fixed point for strong Property (T)},  this representation should admit a non-zero invariant vector. The following proposition allows us to conclude the proof of Theorem \ref{theorem:main result}.

\begin{proposition}
If the representation $H$ admits a non-zero $G$-invariant vector, then every orbit for the action of $G$ on $X$ is bounded.

\end{proposition}

The proof of this Proposition is similar to the proof of Proposition \ref{Proposition : Orbites bornées}.

\subsection{Proof of Theorem \ref{theorem:main result bis}}

In Theorem \ref{theorem:main result}, we prove that every group $G$ hyperbolic relatively to a family of subgroups with infinite index in $G$ admits a unbounded representation of polynomial growth of degree $2$, on every Hilbert space $H_{x}$, for all $x\in X$, where $X$ is the coned-off graph. In particular, it acts on $H_{e}$. Moreover, this representation fixes the continuous linear form $\phi$ defined in \ref{formelinéairecontinue2}.\\

In the first part of \cite{AlvarezLafforguehyperboliclp}, it is explained how to get from this representation an affine representation. The representation $\pi$ preserves the hyperplane $\{ f \in H_{e}~|~ \pi(f)=0 \}$. We remark that the function $\delta_{e}: X \to \mathbb{C} $ defined by :

$$
\delta_{e}(x) = \left\{
    \begin{array}{ll}
        0 & \mbox{if} ~ x \ne e \\
        1 & \mbox{if} ~ x=e
    \end{array}
\right.
$$
is an element of $H_{e}$.\\

Moerover, for all $g \in G$, $\delta_{e}-g.\delta_{e} \in \{ f \in H_{e}~|~ \pi(f)=0 \} $. Therefore, according to \cite{AlvarezLafforguehyperboliclp}, $G$ acts on the affine Hilbert space $\delta_{e}+\{ f \in H_{e}~|~ \pi(f)=0 \}$ with linear part equals to $\pi$ and the cocycle $c(g)=\delta_{e}-g.\delta_{e}$.

To conclude, we remark that :

$$\begin{aligned} 
           \left\| \delta_{e}-g.\delta_{e}  \right\|_{H_{e}}^{2} & = \sum_{n \in \mathbb{N}} (n+1)^{2} \sum_{k=0}^{n} \sum_{i \in J^{n,k,e}} { \biggr\vert \sum_{a \in I_{i}^{n,k,e}} (\delta_{e}-g.\delta_{e})(a) \biggr\lvert }^{2}\\      
            & \geq \sum_{n \in \mathbb{N}} (n+1)^{2}  \sum_{i \in J^{n,n,e}} { \biggr\vert \sum_{a \in I_{i}^{n,n,e}} (\delta_{e}-g.\delta_{e})(a) \biggr\lvert }^{2}\\
            & \geq \sum_{n \in \mathbb{N}} (n+1)^{2} \sum_{a \in S_{x}^{n} } { \vert (\delta_{e}-g.\delta_{e})(a) \lvert }^{2}\\
            & \geq (d(1,g)+1)^{2}+2^{2}~(\text{where } d \text{ denote the coned-off distance}).
\end{aligned}$$

With this fact, we obtain that $\left\|c(g)\right\|_{H_{e}}^{2} \to \infty$ as $g \to \infty$ in the coned-off graph.

This prove Theorem \ref{theorem:main result bis}.

\newpage

\bibliographystyle{alpha}
\bibliography{bibliography}

\end{document}